\documentclass[pdflatex,sn-mathphys-num]{sn-jnl}

\usepackage{graphicx}
\usepackage{amsmath,amssymb,amsfonts}
\usepackage{amsthm}
\usepackage[title]{appendix}
\usepackage{xcolor}
\let\orcidlogo\relax
\usepackage{orcidlink}

\usepackage{mathtools}
\mathtoolsset{showonlyrefs}
\usepackage{bm}
\usepackage{tikz}
\usepackage{placeins}
\usetikzlibrary{calc,hobby,arrows.meta}

\definecolor{surface}{HTML}{F5F5F5}

\theoremstyle{thmstyleone}
\newtheorem*{externaltheorem}{Theorem}
\newtheorem{theorem}{Theorem}
\newtheorem{corollary}{Corollary}
\newtheorem*{externalcorollary}{Corollary}
\newtheorem{lemma}{Lemma}
\newtheorem{proposition}{Proposition}

\theoremstyle{thmstyletwo}

\newtheorem{remark}{Remark}

\theoremstyle{thmstylethree}

\newcommand{\R}{\mathbb{R}}

\DeclareMathOperator{\diag}{diag}
\DeclareMathOperator{\Diag}{Diag}
\DeclareMathOperator{\Cov}{Cov}
\DeclareMathOperator{\Var}{Var}
\DeclareMathOperator{\tr}{tr}
\newcommand{\DPP}{\mathrm{DPP}}
\newcommand{\KL}{\mathrm{KL}}
\newcommand{\E}{\mathbb{E}}

\newcommand{\succsm}{\succeq_{\mathrm{sm}}}
\newcommand{\succism}{\succeq_{\mathrm{ism}}}

\DeclarePairedDelimiter{\ang}{\langle}{\rangle}

\begin{document}

\title[DPP Approximation under Dependence]{Determinantal Point Process Approximation under Positive and Negative Dependence}
\author*[1,2]{%
  \fnm{So} \sur{Anzai}
  \orcidlink{0009-0005-8889-2529}%
}
\email{anzai.so@ism.ac.jp}

\author[1,2]{%
  \fnm{Hideitsu} \sur{Hino}
  \orcidlink{0000-0002-6405-4361}%
}
\email{hino@ism.ac.jp}

\affil[1]{%
  \orgname{Graduate Institute for Advanced Studies, SOKENDAI},
  \orgaddress{%
    \street{Shonan Village},
    \city{Hayama},
    \postcode{240-0193},
    \state{Kanagawa},
    \country{Japan}}}

\affil[2]{%
  \orgname{The Institute of Statistical Mathematics},
  \orgaddress{%
    \street{10-3 Midori-cho},
    \city{Tachikawa},
    \postcode{190-8562},
    \state{Tokyo},
    \country{Japan}}}

\abstract{%
Determinantal point processes (DPPs) are widely used as probabilistic models for diverse random subsets, but their approximation error under model misspecification has not been fully characterized. We study the population-level approximation of a strictly positive target distribution $p^*$ by DPPs under the forward Kullback--Leibler divergence. Using information-geometric analysis and the standard quality--diversity decomposition of an $L$-ensemble kernel, in which the diagonal quality component $Q$ encodes item-specific weights and the diversity component $D$ controls repulsive interactions among items, we show that the quality component can be chosen uniquely to match all first-order inclusion probabilities of $p^*$. The DPP approximation problem therefore reduces to the optimization of the diversity component. This reduction yields a global optimality result for attractively dependent targets: under conditions including weak positive association, the independent product distribution with the same first-order marginals, corresponding to $D = I$, is an optimal DPP approximation. For more general target distributions, positively correlated pairs yield lower bounds on the approximation error. On the repulsive side, a matching of disjoint negatively correlated pairs yields an upper bound on the approximation error, or equivalently a guaranteed improvement over the independent approximation. We further study local optimality around $D = I$ and, more generally, around block-diagonal diversity matrices by analyzing perturbations between their blocks.
}

\maketitle

\section{Introduction}

Determinantal point processes (DPPs) are widely used as probabilistic models for random subsets with repulsive dependence \cite{kulesza-2012-determinantal}. Originally introduced in physics to describe systems of fermions \cite{macchi-1975-coincidence}, DPPs have subsequently found applications in a broad range of fields, including document summarization \cite{kulesza-2011-learning}, recommender systems \cite{gartrell-2016-bayesian}, Bayesian variable selection \cite{kojima-2016-determinantal}, neural spike modeling \cite{snoek-2013-determinantal}, and bioinformatics \cite{batmanghelich-2014-diversifying}. These applications have motivated a variety of methods for learning DPP kernels. Maximum likelihood approaches include EM algorithms \cite{gillenwater-2014-expectation}, fixed-point algorithms \cite{mariet-2015-fixed-point}, minorization--maximization algorithms \cite{kawashima-2023-minorization}, and proximal algorithms \cite{castella-2026-kernel}. Bayesian approaches based on MCMC have also been proposed \cite{affandi-2014-learning}. Other methods use moments or closed-form estimators \cite{urschel-2017-learning,gourieroux-2025-simple}, or impose structured kernels to make learning scalable \cite{mariet-2016-kronecker,gartrell-2017-lowrank}. Efficient sampling from DPPs has also been extensively studied, including classical spectral samplers \cite{hough-2006-determinantal, kulesza-2012-determinantal}, rapidly mixing MCMC algorithms for approximate sampling from $k$-DPPs \cite{anari-2016-mcmc}, and exact samplers with sublinear-time preprocessing \cite{derezinski-2019-exact}.

Despite these developments, a more basic question remains: how well can the DPP family approximate an arbitrary distribution on random subsets? An arbitrary probability distribution on $2^{[N]}$ has $2^N-1$ degrees of freedom, whereas a DPP is described by a symmetric $N \times N$ kernel with at most $N(N+1)/2$ continuous parameters. A DPP is a probability distribution on subsets whose inclusion probabilities are given by principal minors of a positive semidefinite kernel with eigenvalues in $[0,1]$; a formal definition is given in Section~\ref{sec:dpp-preliminary}. Moreover, DPPs satisfy strong negative-dependence properties \cite{borcea-2009-negative}. The DPP family is therefore much smaller than the space of all distributions on $2^{[N]}$, and model misspecification cannot be avoided in general. In particular, repulsive dependence alone does not guarantee exact representation by a DPP.

For a strictly positive target distribution $p^*$, we study the population-level approximation error
\begin{equation}
  \inf_{p \text{ is a DPP on } [N]} \KL(p^*\|p).
\end{equation}
Here ``population-level'' means that the error is defined directly in terms of the target distribution $p^*$, rather than an empirical distribution obtained from finitely many samples. We therefore study model misspecification separately from the additional finite-sample error involved in learning a DPP from data. Because the independent product distribution with the same first-order marginals as $p^*$ is itself a DPP, it provides a canonical baseline. Our central question is when a DPP with nontrivial repulsive interactions can improve upon this independent approximation.

Brunel et al.~\cite{brunel-2017-maximum} analyze the geometry of the expected log-likelihood and the asymptotic behavior of the maximum likelihood estimator when the target distribution is itself a DPP. More generally, information-geometric KL projections and model approximation errors have been studied outside the DPP setting \cite{amari-2001-information, montufar-2012-scaling}. A related problem is to test, from samples, whether an unknown distribution belongs to the DPP family; Gatmiry et al.~\cite{gatmiry-2020-testing} give such a test under the $\ell_1$ distance.

Our main result, established in Section~\ref{sec:reduction}, gives a structural reduction of the DPP approximation problem. We show that the optimization separates into matching the first-order inclusion probabilities of $p^*$ and optimizing the diversity component of a kernel representation, which is responsible for the repulsive interactions of the DPP. This matching is always possible by a unique choice of the quality component. Once the first-order marginals have been matched, the remaining approximation problem depends only on the diversity component. The remaining results characterize when a nontrivial diversity component improves on the independent baseline.

Section~\ref{sec:attractive} considers attractively dependent target distributions. We identify broad conditions under which no such improvement is possible. In particular, for weakly positively associated targets, the independent product distribution is globally optimal among all DPP approximations. Thus, under these conditions, introducing the repulsive structure of a DPP cannot reduce the approximation error. We complement this global result with lower bounds based on positively correlated pairs and a local analysis around the independent approximation.

Section~\ref{sec:repulsive} turns to repulsively dependent targets, for which we obtain constructive improvement guarantees. By combining disjoint negatively correlated pairs, we construct DPP approximations whose improvement over the independent baseline is quantified by pairwise mutual information. We also give a local criterion for determining whether such block-structured approximations can be further improved by interactions between the blocks.

\section{Preliminaries}

\subsection{Determinantal point processes}
\label{sec:dpp-preliminary}
Let $[N]=\{1,2,\dots,N\}$ be a finite ground set. A random subset $Y \subset [N]$ is said to follow a determinantal point process (DPP) with marginal kernel $K \in \R^{N \times N}$ if
\begin{equation}
  \Pr(A \subset Y) = \det K_A \quad (A \subset [N]),
  \label{eq:DPP}
\end{equation}
where $K_A$ denotes the principal submatrix of $K$ indexed by $A$, and we adopt the convention $\det K_\emptyset = 1$. The marginal kernel $K$ is symmetric positive semidefinite and all its eigenvalues lie in $[0,1]$, which we write as $O \preceq K \preceq I$.

If all eigenvalues of $K$ lie in $(0,1)$, i.e., $O \prec K \prec I$, then the DPP admits an $L$-ensemble representation:
\begin{equation}
  p_L(A) \coloneqq \Pr(Y = A) = \frac{\det L_A}{\det(L+I)} \quad (A \subset [N]).
  \label{eq:L-ensemble}
\end{equation}
Here $L$ is also symmetric positive definite, and the marginal kernel $K$ and the $L$-ensemble kernel $L$ are related by
\begin{equation}
  K = L(L+I)^{-1}, \quad L = K(I-K)^{-1}.
  \label{eq:L-K-relation}
\end{equation}

A DPP admitting the representation \eqref{eq:L-ensemble} is called an $L$-ensemble. For a strictly positive target distribution $p^*$, restricting the optimization to $L$-ensembles does not change the approximation error. We therefore write
\begin{equation}
  \KL(p^*\|\DPP) \coloneqq \inf_{p \in \DPP} \KL(p^*\|p) = \inf_{p \in L\text{-ensemble}} \KL(p^*\|p).
\end{equation}

\subsection{Information-Geometric Coordinates for Random-Subset Distributions and DPPs}
\label{sec:ig-coordinates}
A basic probability model on the power set $2^{[N]}$ is the log-linear model
\begin{align}
  p_\theta(A) &= \exp\left[\sum_{\emptyset \neq I \subset [N]} \theta^I 1\{I \subset A\} - \psi(\theta)\right], \label{eq:log-linear} \\
  \psi(\theta) &= \log \sum_{J \subset [N]} \exp\left[\sum_{\emptyset \neq I \subset [N]} \theta^I 1\{I \subset J\}\right].
\end{align}
Here the superscript $I$ is a set-valued index: for each nonempty subset $I \subset [N]$, $\theta^I$ denotes the interaction parameter associated with the joint inclusion of the elements in $I$. For example, $\theta^{\{1,2\}}$ represents the interaction between items $1$ and $2$. The function $\psi$ is the log-partition function of the model, also called the potential function. 

Every strictly positive probability distribution on $2^{[N]}$ admits a unique representation of the form \eqref{eq:log-linear} \cite{amari-2001-information}. Indeed, M\"obius inversion on the Boolean lattice gives
\begin{equation}
    \theta^I = \sum_{J \subset I} (-1)^{|I| - |J|} \log p(J), \quad \emptyset \neq I \subset [N],
\end{equation}
while $\psi(\theta) = -\log p(\emptyset)$. Thus, $\theta=(\theta^I)_{\emptyset \neq I \subset [N]}$ provides a global coordinate system for the space $\mathcal{P}$ of all strictly positive probability distributions on $2^{[N]}$. For $p \in \mathcal{P}$, we write $\theta(p)$ for the $\theta$-coordinate vector of $p$.

The log-linear model is a regular and minimal exponential family. In the standard terminology of information geometry, $\theta$ forms an $e$-affine coordinate system. Moreover, the potential function $\psi$ is strictly convex, and the Fisher information matrix $(g_{I J}(\theta))_{\emptyset \neq I,J \subset [N]}$ is given by
\begin{equation}
  g_{I J}(\theta) = \frac{\partial^2 \psi}{\partial \theta^I \partial \theta^J}(\theta) = \Cov_\theta(1\{I \subset Y\},1\{J \subset Y\}).
\end{equation}

The $m$-affine coordinates $(\eta_I)_{\emptyset \neq I \subset [N]}$ dual to $\theta$ are defined as the Legendre-dual coordinates of $\psi$:
\begin{equation}
  \eta_I(p) = \frac{\partial \psi}{\partial \theta^I}(\theta(p)).
  \label{eq:eta-legendre-dual}
\end{equation}
These coordinates coincide with the expectation parameters:
\begin{equation}
  \eta_I(p) = \E_p[1\{I \subset Y\}].
\end{equation}

Following \cite{amari-2001-information}, we decompose these coordinates according to the cardinality of the indexing subsets:
\begin{align}
  \theta &= (\theta_{(1)},\theta_{(2)},\dots,\theta_{(N)}),\\
  \eta &= (\eta_{(1)},\eta_{(2)},\dots,\eta_{(N)}).
\end{align}
For example,
\begin{equation}
  \theta_{(1)} = (\theta^{\{1\}}, \theta^{\{2\}}, \dots, \theta^{\{N\}}), \quad \theta_{(2)} = (\theta^{\{1,2\}}, \theta^{\{1,3\}}, \dots, \theta^{\{N-1,N\}}), \dots.
\end{equation}
This hierarchy gives rise to the $k$-cut mixed coordinate system, in which coordinates up to order $k$ are represented by $\eta$ and the remaining higher-order coordinates are represented by $\theta$:
\begin{equation}
  \omega_k \coloneqq (\eta_{(1)}, \dots, \eta_{(k)}; \theta_{(k+1)}, \dots, \theta_{(N)}).
  \label{eq:k-cut-mixed-coordinate}
\end{equation}

We use the following standard identity for the KL divergence.

\begin{externaltheorem}[Extended Pythagorean Theorem; {\cite[Theorem~3.7]{amari-2000-methods}}]
  For any three points $p,q,r \in \mathcal{P}$,
  \begin{equation}
    \KL(p\|r) = \KL(p\|q) + \KL(q\|r) + \ang*{\eta(p) - \eta(q), \theta(q) - \theta(r)}, 
    \label{eq:extended-pythagoras}
  \end{equation}
  where
  \begin{equation}
    \ang*{\eta(p) - \eta(q), \theta(q) - \theta(r)} = \sum_{\emptyset \neq I \subset [N]} (\eta_I(p) - \eta_I(q))(\theta^I(q) - \theta^I(r)).
  \end{equation}
\end{externaltheorem}
When the final inner-product term vanishes, this identity is often referred to as the extended Pythagorean identity. Geometrically, the vanishing of the inner product means that the $m$-geodesic connecting $p$ and $q$ is orthogonal to the $e$-geodesic connecting $q$ and $r$. 

Given two points $p,r \in \mathcal{P}$, the extended Pythagorean theorem can also be used to decompose $\KL(p\|r)$ in the $k$-cut mixed coordinate system. The following corollary is a reformulation, in the present notation, of a result in \cite{amari-2001-information}. We include the proof for completeness.
\begin{externalcorollary}[Mixed-coordinate Pythagorean decomposition]
  Suppose that two points $p,r \in \mathcal{P}$ have $k$-cut mixed coordinates
  \begin{align}
    \omega_k(p) &= (\eta_{(1)}, \dots, \eta_{(k)}; \theta_{(k+1)}, \dots, \theta_{(N)}),\\
    \omega_k(r) &= (\tilde{\eta}_{(1)}, \dots, \tilde{\eta}_{(k)}; \tilde{\theta}_{(k+1)}, \dots, \tilde{\theta}_{(N)}).
  \end{align}
  Then there exists a point $q \in \mathcal{P}$ satisfying
  \begin{equation}
    \omega_k(q) = (\eta_{(1)}, \dots, \eta_{(k)}; \tilde{\theta}_{(k+1)}, \dots, \tilde{\theta}_{(N)}),
  \end{equation}
  and this point obeys
  \begin{equation}
    \KL(p\|r) = \KL(p\|q) + \KL(q\|r).
    \label{eq:k-cut-pythagoras}
  \end{equation}
\end{externalcorollary}

\begin{proof}
  The existence of such a point $q \in \mathcal{P}$ is shown in Appendix~\ref{sec:proof-k-cut-pythagoras}. Assuming its existence, we apply the extended Pythagorean theorem to the three points $p,q,r$. The components of $\eta(p) - \eta(q)$ up to order $k$ are zero, while the components of $\theta(q) - \theta(r)$ of order $k+1$ and higher are zero. Hence the inner-product term in \eqref{eq:extended-pythagoras} vanishes.
\end{proof}

In the setting of the mixed-coordinate Pythagorean decomposition, let $\mathcal{E} \subset \mathcal{P}$ be the submanifold on which the coordinates $\theta_{(k+1)}, \dots, \theta_{(N)}$ are fixed at $\tilde{\theta}_{(k+1)}, \dots, \tilde{\theta}_{(N)}$. Then
\begin{equation}
  \KL(p\|q) = \min_{r \in \mathcal{E}} \KL(p\|r).
\end{equation}
Indeed, \eqref{eq:k-cut-pythagoras} holds for any $r \in \mathcal{E}$, and $\KL(q\|r) \geq 0$. For this reason, $q$ is called the $m$-projection, or simply the projection, of $p$ onto $\mathcal{E}$ \cite{amari-2001-information}.

We next express DPPs in these coordinates. We begin with the diagonal scaling of an $L$-ensemble kernel:
\begin{equation}
  L = \sqrt{Q} D \sqrt{Q}.
  \label{eq:quality-vs-diversity}
\end{equation}
Here $Q=\Diag(L_{11}, \dots, L_{NN})$ and $D$ is the symmetric matrix with unit diagonal and off-diagonal entries
\begin{equation}
  D_{ij}=\frac{L_{ij}}{\sqrt{L_{ii} L_{jj}}}.
\end{equation}
This is also known as the quality--diversity decomposition \cite{kulesza-2012-determinantal}. Since
\begin{equation}
  p_L(A) \propto \det L_A = (\det Q_A)(\det D_A),
\end{equation}
$\det Q_A$ increases as the item-wise qualities $Q_{ii}$ of elements in $A$ increase, whereas $\det D_A$ tends to decrease as the similarities $|D_{ij}|$ among elements in $A$ increase. We therefore call $Q$ the quality matrix and $D$ the diversity matrix. The numbers of parameters in $Q$ and $D$ are $N$ and $N(N-1)/2$, respectively. Because we work with $L$-ensemble kernels, both matrices are positive definite. Conversely, any positive diagonal matrix $Q$ and any positive definite symmetric matrix $D$ with unit diagonal determine an $L$-ensemble kernel through \eqref{eq:quality-vs-diversity}.

Using this quality--diversity decomposition, Hino and Yano \cite{hino-2024-embedding} showed that the $\theta$-coordinates of the DPP distribution $p_L$ can be written as
\begin{equation}
  \theta^{\{i\}}(p_L) = \log Q_{ii}, \quad \theta^I(p_L) = \log\det D_I - \sum_{J \subsetneq I;\ |J| \geq 2} \theta^J(p_L) \quad (|I| \geq 2).
  \label{eq:dpp-theta-coordinate}
\end{equation}
Thus, the first-order coordinates are determined only by the quality matrix $Q$, whereas the coordinates of order two or higher are determined recursively only by the diversity matrix $D$. To make these dependencies explicit, we write
\begin{equation}
    \theta_{(1)}(Q) \coloneqq \theta_{(1)}(p_L), \quad \theta_{(2+)}(D) \coloneqq \theta_{(2+)}(p_L),
\end{equation}
where $\theta_{(2+)} = (\theta_{(2)}, \theta_{(3)}, \dots, \theta_{(N)})$. The $\eta$-coordinates, being expectation parameters, are expressed in terms of the marginal kernel $K$ as
\begin{equation}
  \eta_I = \det K_I.
  \label{eq:dpp-eta-coordinate}
\end{equation}

\subsection{Positive and negative dependence}
\label{sec:positive-negative-dependence}

Attractive and repulsive properties of probability distributions on the power set have long been studied under the names of positive and negative dependence, and several mathematical formulations have been proposed. We first recall notions of positive dependence. Let $Y$ be a random subset with distribution $p$, where $\sum_{S \subset [N]} p(S)=1$. The distribution $p$ is said to be positively associated (PA) if, for every pair of increasing functions $f,g:2^{[N]} \to \R$,
\begin{equation}
  \E[f(Y)g(Y)] \geq \E[f(Y)]\E[g(Y)] \iff \Cov(f(Y),g(Y)) \geq 0,
  \label{eq:PA}
\end{equation}
holds \cite{esary-1967-association}. Here a function $f:2^{[N]} \to \R$ is increasing if $S \subset T$ implies $f(S) \leq f(T)$.

A less restrictive property is that of being weakly positively associated (wPA; \cite{burton-1986-invariance}). A distribution $p$ is wPA if, for any disjoint subsets $I,J \subset [N]$ and any pair of increasing functions $f:2^I \to \R$ and $g:2^J \to \R$,
\begin{equation}
  \Cov(f(Y \cap I), g(Y \cap J)) \geq 0.
  \label{eq:wPA}
\end{equation}
Clearly, $\text{(PA)} \Rightarrow \text{(wPA)}$. A sufficient condition for these properties is the positive lattice condition (PLC; \cite{fortuin-1971-correlation}), which requires that
\begin{equation}
  p(S)p(T) \leq p(S \cup T) p(S \cap T) \quad (S,T \subset [N]).
  \label{eq:PLC}
\end{equation}
Taking logarithms gives
\begin{equation}
  \log p(S) + \log p(T) \leq \log p(S \cup T) + \log p(S \cap T),
  \label{eq:log-supermodular}
\end{equation}
which means that $\log p$ is supermodular. For this reason, PLC is also called log-supermodularity. Compared with PA and wPA, which quantify over pairs of increasing functions, PLC is a local and more easily verifiable condition because it is stated directly in terms of the four sets $S$, $T$, $S \cup T$, and $S \cap T$.

The weakest positive-dependence condition considered here is pairwise positive correlation (p-PC):
\begin{equation}
  \Pr(i \in Y) \Pr(j \in Y) \leq \Pr(i,j \in Y) \quad (i \neq j).
  \label{eq:p-PC}
\end{equation}
This is obtained from the wPA condition \eqref{eq:wPA} by taking $I=\{i\}$, $J=\{j\}$, $f(S)=1\{i \in S\}$, and $g(S)=1\{j \in S\}$. The implications among these positive-dependence notions are therefore summarized as
\begin{equation}
  \text{(PLC)} \Rightarrow \text{(PA)} \Rightarrow \text{(wPA)} \Rightarrow \text{(p-PC)}.
\end{equation}

Repulsive properties are defined analogously. A distribution $p$ is said to be negatively associated (NA; \cite{joagdev-1983-negative}) if the inequality in \eqref{eq:wPA} is reversed. There is no direct negative counterpart of PA, because $\Cov(f(Y), f(Y)) = \Var(f(Y)) \geq 0$ holds for every function $f:2^{[N]} \to \R$. Similarly, the negative lattice condition (NLC) is obtained by reversing the inequality in \eqref{eq:PLC}; equivalently, $\log p$ is submodular. While the FKG theorem implies $\text{(PLC)} \Rightarrow \text{(PA)}$ for positive dependence \cite{fortuin-1971-correlation}, no analogous implication holds for negative dependence \cite{pemantle-2004-theory}. Finally, pairwise negative correlation (p-NC) is defined by reversing the inequality in \eqref{eq:p-PC}, and $\text{(NA)} \Rightarrow \text{(p-NC)}$ holds.

Among these conditions, PLC and NLC are often used in modeling, as they provide tractable criteria for attractive and repulsive behavior on the power set \cite{iyer-2015-submodular,kawashima-2025-family}. DPPs are especially relevant in this context: they are known to be strongly Rayleigh \cite{borcea-2009-negative}, a property that implies both NA and NLC. Consequently, DPPs satisfy all of the repulsive-dependence properties discussed above.

\section{Reduction of DPP Approximation to the Diversity Component}
\label{sec:reduction}
We begin by separating the roles of the quality and diversity components in the KL approximation problem. The main result reduces the optimization over $L$-ensemble kernels to an optimization over diversity matrices.

\begin{theorem}
  \label{thm:DPP-error-dependence}
  Let $p^*$ be a strictly positive probability distribution on $2^{[N]}$. The approximation error of $p^*$ by DPPs depends only on the diversity component $D$ of the $L$-ensemble kernel: For each diversity matrix $D$, let $\nu_D \in \mathcal{P}$ be the unique distribution whose $1$-cut mixed coordinate satisfies
  \begin{equation}
    \omega_1(\nu_D) = (\eta^*_{(1)};\theta_{(2+)}(D)), \quad \eta^*_{(1)} \coloneqq \eta_{(1)}(p^*).
  \end{equation}
  Then
  \begin{equation}
    \KL(p^*\|\DPP) = \inf_{D \succ O;\ \diag D = \bm{1}_N} \KL(p^*\|\nu_D).
  \end{equation}
\end{theorem}

\begin{proof}
  For the target distribution $p^*$ and an arbitrary DPP distribution $p_L$, the mixed-coordinate Pythagorean decomposition implies that there exists a distribution $\nu_D \in \mathcal{P}$ satisfying
  \begin{equation}
    \omega_1(\nu_D) = (\eta^*_{(1)};\theta_{(2+)}(p_L)) = (\eta^*_{(1)};\theta_{(2+)}(D)),
  \end{equation}
  and
  \begin{equation}
    \KL(p^*\|p_L) = \KL(p^*\|\nu_D) + \KL(\nu_D\|p_L).
    \label{eq:pythagoras-DPP}
  \end{equation}
  By \eqref{eq:dpp-theta-coordinate}, among the $1$-cut mixed coordinates of the three distributions $p^*$, $\nu_D$, and $p_L$, the only component depending on the quality matrix $Q$ is
  \begin{equation}
    \omega_1(p_L) = (\eta_{(1)}(Q,D); \theta_{(2+)}(D)).
  \end{equation}
  Hence the approximation error can be decomposed as
  \begin{align}
    \KL(p^*\|\DPP)
    &= \inf_L \{\KL(p^*\|\nu_D) + \KL(\nu_D\|p_L)\} \\
    &= \inf_D \left\{\KL(p^*\|\nu_D) + \inf_Q \KL(\nu_D\|p_L)\right\}.
    \label{eq:error-decomposition}
  \end{align}
  Therefore, it remains to show that, for each fixed $D$, there exists a quality matrix $Q$ such that $\eta_{(1)}(Q,D) = \eta^*_{(1)}$. If such a $Q$ exists, then $p_L = \nu_D$ and the second term in \eqref{eq:error-decomposition} attains its minimum value zero.

  We now prove the existence of such a $Q$. Fix $D$. As $Q$ varies, the distributions $p_L$ form a minimal exponential family with sufficient-statistic vector $(1\{i \in Y\})_{i \in [N]}$ and natural-parameter vector $\theta_{(1)}(Q) = (\theta^{\{i\}}(Q))_{i \in [N]}$:
  \begin{align}
    p_L(A) &= \exp\left[\log\det D_A + \sum_{i \in [N]} \theta^{\{i\}}(Q)1\{i \in A\} - \psi_D(\theta_{(1)}(Q))\right], \label{eq:exp-family-theta1} \\
    \psi_D(\theta_{(1)}(Q)) &= \log \sum_{J \subset [N]} \exp\left(\log\det D_J + \sum_{j \in J} \theta^{\{j\}}(Q)\right).
  \end{align}
  Because $D \succ O$, every principal minor $\det D_A$ is positive, so the family has full support on $2^{[N]}$. By \eqref{eq:dpp-theta-coordinate}, varying $Q$ allows $\theta_{(1)}$ to range over all of $\R^N$. Therefore, the exponential family \eqref{eq:exp-family-theta1} is regular. Its mean-parameter space is
  \begin{equation}
    \mathrm{conv} \{{(1\{i \in A\})_{i \in [N]} \mid A \subset [N]\}} = [0,1]^N.
  \end{equation}
  By the standard theory of exponential families, specifically \cite[Proposition~3.2 and Theorem~3.3]{wainwright-2008-graphical}, the gradient map
  \begin{equation}
    \nabla_{\theta_{(1)}}\psi_D: \R^N \to (0,1)^N, \quad \theta_{(1)} \mapsto \nabla_{\theta_{(1)}} \psi_D(\theta_{(1)}) = \eta_{(1)}
  \end{equation}
  is bijective. Since $p^*$ is strictly positive, $\eta^*_{(1)} \in (0,1)^N$, and hence there exists a unique $\theta_{(1)}$ corresponding to $\eta^*_{(1)}$. Setting
  \begin{equation}
    Q = \Diag(e^{\theta^{\{1\}}}, \dots, e^{\theta^{\{N\}}})
  \end{equation}
  gives the desired quality matrix $Q$ satisfying $\eta_{(1)}(Q,D) = \eta^*_{(1)}$, and the theorem follows.
\end{proof}

As visualized in Figure~\ref{fig:KL-decomposition}, Theorem~\ref{thm:DPP-error-dependence} reduces the original optimization over $(Q,D)$ to an optimization over $D$ alone. The remaining question is how well $D$ can represent the higher-order interaction structure of $p^*$, and whether this representation can reduce the approximation error.

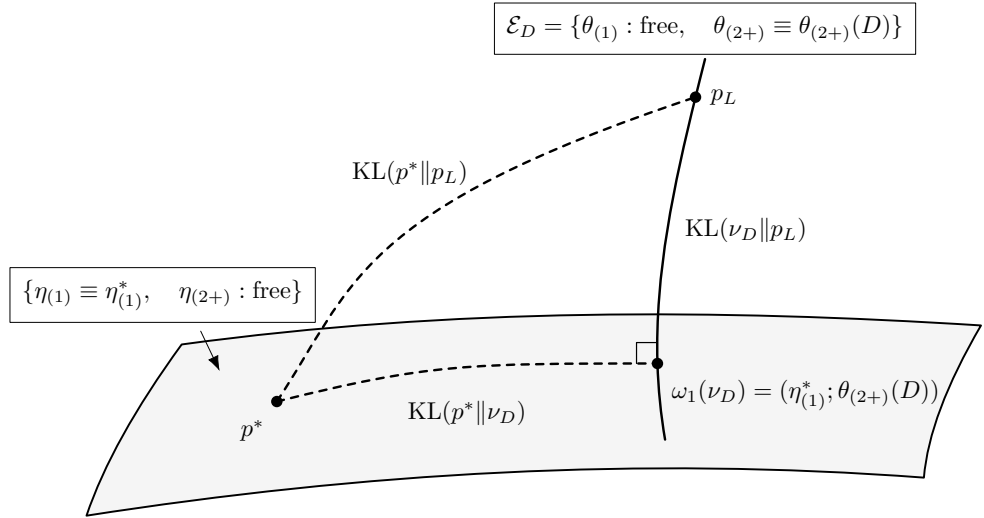
\begin{figure}[!t]
  \centering
  \resizebox{\linewidth}{!}{
  \begin{tikzpicture}[x=0.55cm, y=0.55cm, >=Latex, line cap=round, line join=round]
    \coordinate (A) at (-10,-3.0);
    \coordinate (B) at (12,-2.0);
    \coordinate (C) at (13.5,2.0);
    \coordinate (D) at (-7.5,1.5);
    \coordinate (pstar) at (-5.0,0.0);
    \coordinate (q)     at ( 5.0,1.0);
    \coordinate (p)     at ( 6.0,8.0);
    \coordinate (pafter)  at ($(p)+(0.25,1)$);
    \coordinate (qbefore) at ($(q)+(0.2,-2)$);

    \path[fill=surface, draw=black, line width=0.8pt]
      (A)
      .. controls ($(A)!0.33!(B)+(0,0.8)$) and ($(A)!0.66!(B)+(0,0.9)$) .. (B)
      .. controls ($(B)!0.33!(C)+(-0.4,-0.2)$) and ($(B)!0.66!(C)+(-0.4,-0.2)$) .. (C)
      .. controls ($(C)!0.33!(D)+(0,0.6)$) and ($(C)!0.66!(D)+(0,0.8)$) .. (D)
      .. controls ($(D)!0.33!(A)+(-0.4,-0.2)$) and ($(D)!0.66!(A)+(-0.4,-0.2)$) .. (A)
      -- cycle;

    \draw[use Hobby shortcut, line width=1pt]
      (qbefore) .. (q) .. (p) .. (pafter);

    \draw[dashed, line width=1pt]
      (pstar) .. controls (-3,3) and (-3,5) .. (p);
    \draw[dashed, line width=1pt]
      (pstar) .. controls (-1,1) and (0,1) .. (q);
    \draw[black, line width=0.5pt]
      ($(q)+(-0.55,0)$) -- ($(q)+(-0.55,0.55)$) -- ($(q)+(0,0.55)$);

    \fill (pstar) circle[radius=0.15];
    \fill (q)     circle[radius=0.15];
    \fill (p)     circle[radius=0.15];

    \node[anchor=north east] at ($(pstar)+(-0.10,-0.25)$) {$p^*$};
    \node[anchor=west] at ($(p)+(0.18,0)$) {$p_L$};
    \node[anchor=north west] at ($(q)+(0.15,-0.20)$)
      {$\omega_1(\nu_D) = (\eta^*_{(1)};\theta_{(2+)}(D))$};

    \node at (-1.5,6.0) {$\mathrm{KL}(p^*\|p_L)$};
    \node at (0,-0.3) {$\mathrm{KL}(p^*\|\nu_D)$};
    \node[anchor=west] at (5.5,4.5) {$\mathrm{KL}(\nu_D\|p_L)$};

    \node[draw=black, inner sep=5pt, anchor=center] at ($(pafter)+(0,0.8)$)
      {$\mathcal{E}_D = \{\theta_{(1)}:\text{free}, \quad \theta_{(2+)} \equiv \theta_{(2+)}(D)\}$};
    \node[draw=black, inner sep=5pt, anchor=center, align=left] at ($(pstar)+(-3,2.8)$)
      {$\{\eta_{(1)} \equiv \eta^*_{(1)}, \quad \eta_{(2+)}:\text{free}\}$};

    \draw[-{Latex[length=2.2mm]}]
      ($(pstar)+(-2,1.85)$) -- ($(pstar)+(-1.5,0.8)$);
  \end{tikzpicture}
  }

  \caption{
    Geometric interpretation of Theorem~\ref{thm:DPP-error-dependence}. For each fixed diversity matrix $D$, varying the quality matrix $Q$ moves the DPP distribution $p_L$ within the $e$-flat family $\mathcal{E}_D =\{p\in\mathcal{P} \mid \theta_{(2+)}(p) \equiv \theta_{(2+)}(D)\}$. Its unique intersection $\nu_D$ with the $m$-flat family $\{p\in\mathcal{P} \mid \eta_{(1)}(p) \equiv \eta^*_{(1)}\}$ is the DPP whose first-order marginals match those of $p^*$. The Pythagorean decomposition $\KL(p^*\|p_L) = \KL(p^*\|\nu_D) + \KL(\nu_D\|p_L)$ shows that optimizing over $Q$ eliminates the second term, leaving an optimization over $D$ alone
  }
  \label{fig:KL-decomposition}
\end{figure}

\FloatBarrier

\section{Attractive Dependence: When the Optimal DPP Is Independent}
\label{sec:attractive}
Theorem~\ref{thm:DPP-error-dependence} immediately gives a trivial upper bound by taking $D = I$:
\begin{equation}
  \KL(p^*\|\DPP) \leq \KL(p^*\|\nu_I).
\end{equation}
Here $\nu_I$ is the product of the one-dimensional marginals of $p^*$,
\begin{equation}
  \nu_I = p^*_{\{1\}} \otimes p^*_{\{2\}} \otimes \cdots \otimes p^*_{\{N\}}.
\end{equation}
Since each one-dimensional marginal is Bernoulli, $\nu_I$ is the independent product distribution with the same first-order marginals as $p^*$. Consequently, $\KL(p^*\|\nu_I)$ is the total correlation of the Bernoulli random variables $X_i \coloneqq 1\{i \in Y\}$, $i=1, \dots, N$.

In the $1$-cut mixed coordinate system, this independence, or equivalently the absence of interactions among the elements, is expressed by the vanishing of all interaction coordinates of order two or higher:
\begin{equation}
  \omega_1(\nu_I) = (\eta^*_{(1)}; \theta_{(2+)}(I)) = (\eta^*_{(1)}; 0),
\end{equation}
where the last equality follows from \eqref{eq:dpp-theta-coordinate}.

If DPP modeling is appropriate for approximating $p^*$, one would expect that using a nontrivial diversity component $D \neq I$ improves on this independent approximation, namely that $\KL(p^*\|\nu_D) < \KL(p^*\|\nu_I)$ for some $D$.

\subsection{Global Optimality of the Independent Approximation}
\label{sec:global-optimality}
We first show that if the true distribution $p^*$ has sufficiently strong attractive dependence, then no nontrivial diversity matrix $D \neq I$ can improve the approximation error over $\nu_I$. To state the result precisely, we recall stochastic orders based on supermodular functions \cite{muller-2002-comparison}. A function $h:2^{[N]} \to \R$ is supermodular if, for all
$S,T \subset [N]$,
\begin{equation}
  h(S \cap T) + h(S \cup T) \geq h(S) + h(T).
  \label{eq:supermodular}
\end{equation}
For $p,q \in \mathcal{P}$, we write $p \succsm q$ if
\begin{equation}
  \E_p[h(Y)] \geq \E_q[h(Y)]
\end{equation}
holds for every supermodular function $h:2^{[N]} \to \R$. We also write $p \succism q$ when the same inequality holds for every increasing supermodular function.

\begin{remark}
  \label{rem:ism-sm-equivalence}
  In general, $p \succsm q$ implies $p \succism q$, whereas the converse need not hold. The converse does hold, however, when $\E_p[|Y|] = \E_q[|Y|]$. This condition is satisfied, in particular, by $p = p^*$ and $q = \nu_I$, since $\nu_I$ has the same first-order marginals as $p^*$.

  We now prove the general statement. Let $h$ be any supermodular function, and choose $C \geq 0$ such that
  \begin{equation}
    C \geq \max_{S \subsetneq T \subset [N]} \frac{h(S)-h(T)}{|T|-|S|},
  \end{equation}
  and define $h'(S) \coloneqq h(S) + C|S|$. For every $S \subset T \subset [N]$,
  \begin{equation}
    h'(T) - h'(S) = h(T) - h(S) + C(|T| - |S|) \geq 0,
  \end{equation}
  and hence $h'$ is increasing. Moreover, since $S \mapsto |S|$ is supermodular, $h'$ is also supermodular. Therefore, if $p \succism q$, then
  \begin{equation}
    \E_p [h'(Y)] - \E_q [h'(Y)] \geq 0,
  \end{equation}
  and the assumption $\E_p[|Y|] = \E_q[|Y|]$ implies that
  \begin{equation}
    \E_p [h(Y)] - \E_q [h(Y)] \geq 0,
  \end{equation}
  as well. Thus, $p \succsm q$.
\end{remark}

The following lemma identifies an increasing supermodular function that is naturally induced by the diversity component of a DPP.
\begin{lemma}
  \label{lem:logdet-sm}
  Let $D$ be the diversity component of a DPP $L$-ensemble kernel, with $\diag D = \bm{1}_N$ and $D \succ O$. Then the function $S \mapsto -\log\det D_S$ on $2^{[N]}$ is increasing and supermodular.
\end{lemma}
\begin{proof}
  For any $S,T \subset [N]$, Koteljanskii's inequality, also known as the Hadamard--Fischer inequality \cite[Theorem~7.8.9]{horn-2012-matrix}, gives
  \begin{equation}
    \det D_{S \cup T}\det D_{S \cap T} \leq \det D_S\det D_T.
  \end{equation}
  Taking logarithms and reversing the sign yields
  \begin{equation}
    -\log\det D_{S \cup T} - \log\det D_{S \cap T} \geq -\log\det D_S - \log\det D_T,
  \end{equation}
  which is supermodularity.

  It remains to show monotonicity. Let $S \subset T$. Applying Koteljanskii's inequality to $S$ and $T \setminus S$ gives
  \begin{equation}
    \det D_T \det D_\emptyset \leq \det D_S \det D_{T\setminus S}.
  \end{equation}
  Moreover, since $\diag D=\bm{1}_N$, for any $i \in T \setminus S$ we have
  \begin{equation}
    \det D_{T \setminus S} \det D_{\emptyset} \leq \det D_{\{i\}} \det D_{(T \setminus S) \setminus \{i\}} = \det D_{(T \setminus S) \setminus \{i\}}.
  \end{equation}
  Therefore, by induction, $\det D_{T \setminus S} \leq \det D_\emptyset = 1$. This proves $\det D_T \leq \det D_S$, and hence
  \begin{equation}
    -\log\det D_T \geq -\log\det D_S.
  \end{equation}
  Thus $S \mapsto -\log\det D_S$ is increasing.
\end{proof}

We can now state the main global non-improvement result for attractive distributions.
\begin{theorem}
  \label{thm:nonvalidity-ism}
  If $p^* \succism \nu_I$, then the diversity component $D = I$ is optimal for approximating $p^*$ by DPPs:
  \begin{equation}
    \KL(p^*\|\DPP) = \KL(p^*\|\nu_I).
  \end{equation}
  The corresponding kernels are
  \begin{align}
    K &= \Diag(\eta^*_{\{1\}},\dots,\eta^*_{\{N\}}), \label{eq:nonvalid-opt-K}\\
    L &= \Diag\left(\frac{\eta^*_{\{1\}}}{1-\eta^*_{\{1\}}},\dots,\frac{\eta^*_{\{N\}}}{1-\eta^*_{\{N\}}}\right). \label{eq:nonvalid-opt-L}
  \end{align}
\end{theorem}

Before proving the theorem, we record a more intuitive sufficient condition in terms of the positive-dependence notions introduced in Section~\ref{sec:positive-negative-dependence}.
\begin{corollary}
  \label{cor:nonvalidity-p-dependence}
  If $p^*$ is weakly positively associated, then $D = I$ is optimal for DPP approximation of $p^*$. Consequently, the same conclusion holds if $p^*$ is positively associated or satisfies the positive lattice condition, equivalently log-supermodularity. The optimal kernels are given by \eqref{eq:nonvalid-opt-K} and \eqref{eq:nonvalid-opt-L}.
\end{corollary}

To prove Corollary~\ref{cor:nonvalidity-p-dependence}, identify each subset $S \subset [N]$ with its binary vector $\bm{1}_S = (1\{1 \in S\}, \dots, 1\{N \in S\}) \in \{0,1\}^N$. For $x,y \in \{0,1\}^N$, define
\begin{align}
  x \land y &\coloneqq (\min\{x_1, y_1\}, \dots, \min\{x_N, y_N\}), \label{eq:vector-and}\\
  x \lor y &\coloneqq (\max\{x_1, y_1\}, \dots, \max\{x_N, y_N\}). \label{eq:vector-or}
\end{align}
Then $S \cap T$ and $S \cup T$ correspond to $\bm{1}_S \land \bm{1}_T$ and $\bm{1}_S \lor \bm{1}_T$, respectively. The set inclusion order corresponds to the partial order
\begin{equation}
  x \preceq y \quad \iff \quad x_i \leq y_i \quad \text{for all } i=1, \dots, N.
\end{equation}
In this notation, supermodularity is the condition that $h:\{0,1\}^N \to \R$ satisfies
\begin{equation}
  h(x \land y) + h(x \lor y) \geq h(x) + h(y)
  \label{eq:vector-supermodular}
\end{equation}
for all $x,y \in \{0,1\}^N$. A random subset $Y$ is likewise identified with the random binary vector $X \coloneqq \bm{1}_Y$. Weak positive association can then be written as follows: for any disjoint $I,J \subset [N]$ and any increasing functions $f$ of $(x_i)_{i\in I}$ and $g$ of $(x_j)_{j\in J}$,
\begin{equation}
  \Cov(f((X_i)_{i\in I}), g((X_j)_{j\in J})) \geq 0.
  \label{eq:vector-wPA}
\end{equation}

The following theorem is the key ingredient.
\begin{externaltheorem}[Independent-product comparison in supermodular order; {\cite[Theorem~1]{christofides-2004-connection}}]
  Suppose that a random vector $X \in \{0,1\}^N$ satisfies weak positive association in the sense of \eqref{eq:vector-wPA}. Let $Z=(Z_1, \dots, Z_N)$ be a random vector with independent components such that $Z_i$ has the same distribution as $X_i$ for each $i$. Then $X \succsm Z$. If $X$ is negatively associated, then $Z \succsm X$.
\end{externaltheorem}

This theorem was proved in \cite{christofides-2004-connection} for random vectors in $\R^N$, where $\land$, $\lor$, supermodularity, and weak positive association are defined analogously. Since $\{0,1\}^N \subset \R^N$, the stated binary version follows immediately. For completeness, a simple discrete proof specialized to binary vectors is given in Appendix~\ref{sec:discrete-wPA-sm-ineq}.

\begin{proof}[Proof of Corollary~\ref{cor:nonvalidity-p-dependence}]
  Let $X^*$ be the binary random vector corresponding to a random subset with distribution $p^*$. Let $Z^*=(Z^*_1, \dots, Z^*_N)$ have independent components with $Z^*_i$ distributed as $X^*_i$. Then $Z^*$ corresponds exactly to $\nu_I$. If $X^*$ is weakly positively associated, the independent-product comparison theorem implies $X^* \succsm Z^*$, or equivalently $p^* \succsm \nu_I$. The result follows from Theorem~\ref{thm:nonvalidity-ism}.
\end{proof}

\begin{proof}[Proof of Theorem~\ref{thm:nonvalidity-ism}]
  Applying the extended Pythagorean theorem to $p^*$, $\nu_I$, and $\nu_D$ gives
  \begin{align}
    \KL(p^*\|\nu_D) &= \KL(p^*\|\nu_I) + \KL(\nu_I\|\nu_D) + f^*(D), \label{eq:KL-decomposition-pD} \\
    f^*(D) &= \ang*{\eta^* - \eta(\nu_I), \theta(\nu_I) - \theta(\nu_D)}.
  \end{align}
  Since $\eta_{(1)}(\nu_I) = \eta^*_{(1)}$ and $\theta_{(2+)}(\nu_I) = 0$, the inner-product term becomes
  \begin{equation}
    f^*(D) = \ang*{\theta_{(2+)}(D), \eta_{(2+)}(\nu_I) - \eta^*_{(2+)}}.
  \end{equation}
  Using $\eta_J = \E[1\{J\subset Y\}]$, we can rewrite this term as
  \begin{align}
    f^*(D) &= \sum_{|J|\geq 2} \theta^J(D) \left(\E_{\nu_I}[1\{J\subset Y\}] - \E_{p^*}[1\{J\subset Y\}]\right)\\
    &= \E_{\nu_I}\left[\sum_{|J|\geq 2} \theta^J(D) 1\{J\subset Y\}\right] - \E_{p^*}\left[\sum_{|J|\geq 2} \theta^J(D) 1\{J\subset Y\}\right]\\
    &= \E_{p^*}[-\log\det D_Y] - \E_{\nu_I}[-\log\det D_Y].
    \label{eq:inner-prod-logdet}
  \end{align}
  The last equality follows from \eqref{eq:dpp-theta-coordinate}, which implies
  \begin{equation}
    \sum_{J \subset I;\ |J|\geq 2} \theta^J = \sum_{|J|\geq 2} \theta^J 1\{J\subset I\} = \log\det D_I.
  \end{equation}
  By Lemma~\ref{lem:logdet-sm}, the function $S \mapsto - \log\det D_S$ is increasing and supermodular. Therefore, if $p^* \succism \nu_I$, then \eqref{eq:inner-prod-logdet} is nonnegative. Hence \eqref{eq:KL-decomposition-pD} yields
  \begin{equation}
    \KL(p^*\|\nu_D) \geq \KL(p^*\|\nu_I) + \KL(\nu_I\|\nu_D) \geq \KL(p^*\|\nu_I).
  \end{equation}
  Taking the infimum over $D$ on both sides, Theorem~\ref{thm:DPP-error-dependence} gives $\KL(p^*\|\DPP) \geq \KL(p^*\|\nu_I)$. Equality holds only when $\KL(\nu_I\|\nu_D) = 0$. This implies $\nu_I = \nu_D$, that is, $D = I$.

  In this case the $L$-ensemble kernel has only a quality component, so $L = Q$ is diagonal and $K = Q(I+Q)^{-1}$ by \eqref{eq:L-K-relation}. From \eqref{eq:dpp-eta-coordinate},
  \begin{equation}
    \eta^*_{\{i\}} = \det K_{\{i\}} = K_{ii} = \frac{Q_{ii}}{1 + Q_{ii}},
  \end{equation}
  which gives \eqref{eq:nonvalid-opt-K} and \eqref{eq:nonvalid-opt-L}.
\end{proof}

Section~\ref{sec:numerical-example} numerically illustrates Theorem~\ref{thm:nonvalidity-ism} and Corollary~\ref{cor:nonvalidity-p-dependence} for a concrete two-parameter family of distributions.

\subsection{Lower Bounds from Positively Correlated Pairs}
\label{sec:DPP-error-lb-for-pPC}

Pairwise positive correlation is a weaker attractive property implied by weak positive association, but by itself it does not necessarily imply $p^* \succism \nu_I$. Indeed, one can construct examples satisfying pairwise positive correlation for which a nontrivial DPP approximation with $D \neq I$ improves on $\nu_I$; see Section~\ref{sec:numerical-example}. Nevertheless, because every DPP satisfies pairwise negative correlation, if the true distribution $p^*$ has a pair with strict positive correlation, namely
\begin{equation}
  \eta^*_{\{i\}} \eta^*_{\{j\}} < \eta^*_{\{i,j\}},
\end{equation}
then $p^*$ cannot itself be a DPP. In this case, the DPP approximation error has a strictly positive lower bound.

\begin{lemma}
  \label{lem:KL-dpi}
  For any $S \subset [N]$, let $p^*_S$ be the marginal distribution of $p^*$ on $2^S$, and let $\DPP_S$ denote the class of DPPs on $S$. Then
  \begin{equation}
    \KL(p^*\|\DPP) \geq \KL(p^*_S\|\DPP_S).
  \end{equation}
\end{lemma}
\begin{proof}
  Let $\pi:2^{[N]} \to 2^S$ be the marginalization map $\pi(A) = A \cap S$. By the data processing inequality for KL divergence \cite{cover-1999-elements}, for any $p \in \DPP$,
  \begin{equation}
    \KL(p^*\|p) \geq \KL(\pi_\#(p^*)\|\pi_\#(p)) = \KL(p^*_S\|p_S).
  \end{equation}
  The marginal $p_S$ is also a DPP on $S$; if $K$ is the marginal kernel of $p$, then $K_S$ is the marginal kernel of $p_S$ \cite{kulesza-2012-determinantal}. Therefore, as $p$ ranges over DPPs on $[N]$, $p_S$ ranges over DPPs on $S$. Taking the infimum over $p \in \DPP$ gives the claim.
\end{proof}

\begin{proposition}
  \label{prop:KL-lb-pPC}
  Suppose that there exists at least one pair $i \neq j$ such that
  \begin{equation}
    \eta^*_{\{i\}}\eta^*_{\{j\}} < \eta^*_{\{i,j\}}.
  \end{equation}
  Then
  \begin{align}
    \KL(p^*\|\DPP) &\geq \max_{(i,j):\medspace \eta^*_{\{i\}}\eta^*_{\{j\}} < \eta^*_{\{i,j\}}} I_{p^*}(X_i;X_j)\\
    &\geq 8 \max_{(i,j): \medspace \eta^*_{\{i\}}\eta^*_{\{j\}} < \eta^*_{\{i,j\}}} (\eta^*_{\{i,j\}} - \eta^*_{\{i\}}\eta^*_{\{j\}})^2 > 0,
  \end{align}
  where $X_i = 1\{i\in Y\}$ and $I_{p^*}(X_i;X_j)$ denotes the mutual information between $X_i$ and $X_j$ under $p^*$.
\end{proposition}
\begin{proof}
  Fix any pair $i\neq j$ satisfying $\eta^*_{\{i\}}\eta^*_{\{j\}} < \eta^*_{\{i,j\}}$. By Lemma~\ref{lem:KL-dpi},
  \begin{equation}
    \KL(p^*\|\DPP) \geq \KL(p^*_{\{i,j\}}\|\DPP_{\{i,j\}}).
  \end{equation}
  The two-dimensional marginal $p^*_{\{i,j\}}$ satisfies the positive lattice condition. The only nontrivial PLC inequality \eqref{eq:PLC} is
  \begin{equation}
    p^*_{\{i,j\}}(\{i\})p^*_{\{i,j\}}(\{j\}) \leq p^*_{\{i,j\}}(\{i,j\})p^*_{\{i,j\}}(\emptyset).
    \label{eq:PLC-pair}
  \end{equation}
  Since
  \begin{align}
    p^*_{\{i,j\}}(\{i,j\}) &= \eta^*_{\{i,j\}},\\
    p^*_{\{i,j\}}(\{i\}) &= \eta^*_{\{i\}} - \eta^*_{\{i,j\}},\\
    p^*_{\{i,j\}}(\{j\}) &= \eta^*_{\{j\}} - \eta^*_{\{i,j\}},\\
    p^*_{\{i,j\}}(\emptyset) &= 1 - \eta^*_{\{i\}} - \eta^*_{\{j\}} + \eta^*_{\{i,j\}},
  \end{align}
  subtracting the left-hand side of \eqref{eq:PLC-pair} from the right-hand side gives
  \begin{equation}
    \eta^*_{\{i,j\}} - \eta^*_{\{i\}}\eta^*_{\{j\}} > 0.
  \end{equation}
  Hence Corollary~\ref{cor:nonvalidity-p-dependence} implies
  \begin{equation}
    \KL(p^*_{\{i,j\}}\|\DPP_{\{i,j\}}) = \KL(p^*_{\{i,j\}}\|p^*_{\{i\}}\otimes p^*_{\{j\}}) = I_{p^*}(X_i;X_j).
  \end{equation}
  To obtain the explicit lower bound, fix such a pair $(i,j)$ and write
  \begin{equation}
      \delta_{ij} \coloneqq \eta^*_{\{i,j\}} - \eta^*_{\{i\}}\eta^*_{\{j\}} > 0.
  \end{equation}
  Then
  \begin{equation}
      \|p^*_{\{i,j\}} - p^*_{\{i\}} \otimes p^*_{\{j\}}\|_1 = 4 \delta_{ij}.
  \end{equation}
  Hence Pinsker's inequality gives
  \begin{equation}
    I_{p^*}(X_i;X_j) = \KL(p^*_{\{i,j\}}\|p^*_{\{i\}} \otimes p^*_{\{j\}}) \geq \frac{1}{2} \|p^*_{\{i,j\}} - p^*_{\{i\}} \otimes p^*_{\{j\}}\|_1^2 = 8\delta_{ij}^2 > 0.
  \end{equation}
  Taking the maximum over all strictly positively correlated pairs completes the proof.
\end{proof}

\subsection{Pairwise Positive Correlation and Local Optimality}
The preceding discussion focused on the effect of positive correlation for a single pair. The condition that all pairs are positively correlated, namely pairwise positive correlation in \eqref{eq:p-PC}, is also closely related to the local approximation behavior of the diversity parameter around $D = I$.

We use the following perturbation expansion for $\nu_D$ in Theorem~\ref{thm:DPP-error-dependence}. Fix a diversity matrix $D$, and perturb it in a symmetric direction $E$ satisfying $E_{ii} = 0$ and $\|E\|_F = 1$ by setting
\begin{equation}
  D_t = D + tE,
\end{equation}
where $t>0$ is sufficiently small so that $D_t \succ O$. Applying the extended Pythagorean theorem to $p^*$, $\nu_D$, and $\nu_{D_t}$, and proceeding as in the derivation of \eqref{eq:inner-prod-logdet}, yields the expansion
\begin{equation}
  \KL(p^*\|\nu_{D_t}) = \KL(p^*\|\nu_D) + \KL(\nu_D\|\nu_{D_t}) + \sum_{k\geq 1} \frac{(-1)^k}{k} t^k \Delta_{k,D}(E),
  \label{eq:KL-perturbation}
\end{equation}
whenever $t\|D_J^{-1}E_J\|_F < 1$ for all $\emptyset \neq J \subset[N]$. Here
\begin{equation}
  \Delta_{k,D}(E) \coloneqq \E_{p^*}[\tr\{(D_Y^{-1} E_Y)^k\}] - \E_{\nu_D}[\tr\{(D_Y^{-1} E_Y)^k\}].
  \label{eq:perturbation-coeff}
\end{equation}
When $Y = \emptyset$, the matrix $D_Y^{-1} E_Y$ is interpreted as the $0 \times 0$ matrix, and we adopt the convention $\tr\{(D_Y^{-1} E_Y)^k\} = 0$. A detailed derivation of \eqref{eq:KL-perturbation} is given in Appendix~\ref{sec:derivation-KL-perturbation}. This expansion implies the following local characterization.

\begin{proposition}
  \label{prop:local-pPC}
  In the minimization of $\KL(p^*\|\nu_D)$ in Theorem~\ref{thm:DPP-error-dependence}, if $D = I$ is a local minimizer, then $p^*$ satisfies pairwise positive correlation. Conversely, if $p^*$ satisfies strict pairwise positive correlation, that is,
  \begin{equation}
    \eta^*_{\{i,j\}} > \eta^*_{\{i\}}\eta^*_{\{j\}} \quad \text{for all } i \neq j,
  \end{equation}
  then $D = I$ is a local minimizer.
\end{proposition}
\begin{proof}
  Setting $D = I$ in \eqref{eq:KL-perturbation} gives
  \begin{align}
    \KL(p^*\|\nu_{I+tE}) &= \KL(p^*\|\nu_I) + \KL(\nu_I\|\nu_{I+tE}) + \sum_{k\geq 1} \frac{(-1)^k}{k} t^k \Delta_{k,I}(E), \label{eq:KL-perturbation-I}\\
    \Delta_{k,I}(E) &= \E_{p^*}[\tr(E_Y^k)] - \E_{\nu_I}[\tr(E_Y^k)].
  \end{align}
  Since $E$ has zero diagonal, $\tr E_Y = 0$ and hence $\Delta_{1,I}(E) = 0$. The first potentially nonzero term is therefore of order $t^2$. For this second-order coefficient,
  \begin{align}
    \E_{p^*}[\tr(E_Y^2)] &= \E_{p^*}\left[\sum_{\substack{i,j \in Y \\ i \neq j}} E_{ij}^2\right] = \sum_{\substack{i,j \in [N] \\ i \neq j}} E_{ij}^2 \E_{p^*}[1\{\{i,j\} \subset Y\}] = \sum_{\substack{i,j \in [N] \\ i \neq j}} E_{ij}^2 \eta^*_{\{i,j\}} ,\\
    \E_{\nu_I}[\tr(E_Y^2)] &= \sum_{\substack{i,j \in [N] \\ i \neq j}} E_{ij}^2 \eta^*_{\{i\}}\eta^*_{\{j\}}.
  \end{align}
  Writing the terms of order $t^3$ and higher as $R_E(t)$, we obtain
  \begin{equation}
    \KL(p^*\|\nu_{I+tE}) - \KL(p^*\|\nu_I) = \frac{t^2}{2}\sum_{\substack{i,j \in [N] \\ i \neq j}} E_{ij}^2 (\eta^*_{\{i,j\}} - \eta^*_{\{i\}}\eta^*_{\{j\}}) + R_E(t) + \KL(\nu_I\|\nu_{I+tE}).
    \label{eq:KL-perturbation-2nd}
  \end{equation}
  By Lemma~\ref{lem:KL-perturbation-order}, $\KL(\nu_I\|\nu_{I+tE}) = O(t^4)$, while $R_E(t) = O(t^3)$. Thus
  \begin{equation}
    \KL(p^*\|\nu_{I+tE}) - \KL(p^*\|\nu_I) = \frac{t^2}{2} \sum_{\substack{i,j \in [N] \\ i \neq j}} E_{ij}^2 (\eta^*_{\{i,j\}} - \eta^*_{\{i\}}\eta^*_{\{j\}}) + O(t^3).
  \end{equation}
  If $D = I$ is a local minimizer, this expression must be nonnegative for every direction $E$. By choosing directions supported on a single off-diagonal pair, we obtain
  \begin{equation}
    \eta^*_{\{i,j\}} - \eta^*_{\{i\}}\eta^*_{\{j\}} \geq 0 \quad \text{for all } i \neq j,
  \end{equation}
  which is pairwise positive correlation.

  Conversely, suppose that the strict inequalities hold for all pairs and set
  \begin{equation}
    \delta \coloneqq \min_{i\neq j} (\eta^*_{\{i,j\}} - \eta^*_{\{i\}}\eta^*_{\{j\}}) > 0.
  \end{equation}
  Since $\KL(\nu_I\|\nu_{I+tE}) \geq 0$, \eqref{eq:KL-perturbation-2nd} implies
  \begin{equation}
    \KL(p^*\|\nu_{I+tE}) - \KL(p^*\|\nu_I) \geq \frac{\delta t^2}{2} \sum_{\substack{i,j \in [N] \\ i \neq j}} E_{ij}^2 + R_E(t) = \frac{\delta t^2}{2} + R_E(t),
    \label{eq:KL-perturbation-lb}
  \end{equation}
  where we used $\|E\|_F = 1$ and $E_{ii} = 0$. Moreover,
  \begin{equation}
    |R_E(t)| \leq \sum_{k \geq 3} \frac{t^k}{k} |\Delta_{k,I}(E)| \leq 2 \sum_{k \geq 3} \frac{t^k}{k} \leq \frac{2}{3} \frac{t^3}{1-t},
  \end{equation}
  where the second inequality uses Lemma~\ref{lem:perturbation-coeff-ub}. Hence the right-hand side of \eqref{eq:KL-perturbation-lb} is strictly positive, uniformly over all admissible directions $E$, whenever $0 < t < 3\delta/(4 + 3\delta)$. Therefore $D = I$ is a local minimizer.
\end{proof}

\section{Repulsive Dependence: Guaranteed Improvement over Independence}
\label{sec:repulsive}
Even when the target distribution $p^*$ is not itself a DPP, one can derive nontrivial upper bounds on the approximation error $\KL(p^*\|\DPP)$ if $p^*$ has a certain degree of repulsive dependence.

\subsection{Upper Bounds via Repulsive Matchings}

We first record two elementary facts that will be used to construct block-diagonal DPP approximations.

\begin{lemma}
  \label{lem:DPP-exact-representation}
  If $N = 1$, every strictly positive distribution on $2^{[N]}$ is a DPP. If $N = 2$, the same holds for every strictly positive distribution $p^*$ on $2^{[N]}$ satisfying
  \begin{equation}
    \eta^*_{\{1,2\}} \leq \eta^*_{\{1\}}\eta^*_{\{2\}}.
  \end{equation}
\end{lemma}
\begin{proof}
  The case $N = 1$ is immediate: taking $K = \eta^*_{\{1\}}$ gives $0 < K < 1$ because $p^*$ is strictly positive.
  
  For $N = 2$, it suffices to take the marginal kernel $K$ with
  \begin{equation}
    K_{ii} = \eta^*_{\{i\}}, \quad K_{12} = \pm \sqrt{\eta^*_{\{1\}}\eta^*_{\{2\}} - \eta^*_{\{1,2\}}}.
  \end{equation}
  The assumed pairwise negative correlation ensures that $K_{12}$ is real. Moreover,
  \begin{equation}
    \det K = \eta^*_{\{1,2\}} > 0,
  \end{equation}
  while
  \begin{equation}
    \det(I - K) = 1 - \eta^*_{\{1\}} - \eta^*_{\{2\}} + \eta^*_{\{1,2\}} = p^*(\emptyset) > 0.
  \end{equation}
  Since $0 < \eta^*_{\{i\}} < 1$ for $i = 1,2$, both $K$ and $I - K$ are positive definite. Therefore, $O \prec K \prec I$.
\end{proof}

\begin{lemma}
  \label{lem:block-diagonal-DPP}
  Let $B_1, \dots, B_k$ be a partition of $[N]$. Suppose that the $L$-ensemble kernel of a DPP is block diagonal,
  \begin{equation}
    L = \Diag(L_{B_1}, L_{B_2}, \dots, L_{B_k}).
  \end{equation}
  Then $p_L$ factorizes as the product of the DPP distributions $p_{L_{B_b}}$ on the blocks $B_b$:
  \begin{equation}
    p_L = p_{L_{B_1}} \otimes p_{L_{B_2}} \otimes \cdots \otimes p_{L_{B_k}}.
    \label{eq:DPP-block-product}
  \end{equation}
  Moreover, let $p^*_{B_b}$ be the marginal distribution of $p^*$ on block $B_b$, and define
  \begin{equation}
    \widetilde{p}^{\,*} = p^*_{B_1} \otimes \cdots \otimes p^*_{B_k}.
  \end{equation}
  Then
  \begin{equation}
    \KL(p^*\|p_L) = \KL(p^*\|\widetilde{p}^{\,*}) + \sum_{b=1}^k \KL(p^*_{B_b}\|p_{L_{B_b}}).
    \label{eq:KL-block-decomposition-DPP}
  \end{equation}
\end{lemma}
\begin{proof}
  For any $A\subset[N]$, block diagonality gives
  \begin{equation}
    p_L(A) = \frac{\det L_A}{\det(L+I)} = \prod_{b = 1}^k \frac{\det (L_{B_b})_{A \cap B_b}}{\det(L_{B_b} + I)} = \prod_{b = 1}^k p_{L_{B_b}}(A \cap B_b),
  \end{equation}
  which proves \eqref{eq:DPP-block-product}.

  More generally, for a product distribution $q = q_{B_1} \otimes \cdots \otimes q_{B_k}$,
  \begin{align}
    \KL(p^*\|q) &= \sum_{J \subset [N]} p^*(J) \left[\log \frac{p^*(J)}{\widetilde{p}^{\,*}(J)} + \log \frac{\widetilde{p}^{\,*}(J)}{q(J)}\right]\\
    &= \KL(p^*\|\widetilde{p}^{\,*}) + \sum_{b = 1}^k \sum_{J \subset B_b} p^*_{B_b}(J) \log \frac{p^*_{B_b}(J)}{q_{B_b}(J)}\\
    &= \KL(p^*\|\widetilde{p}^{\,*}) + \sum_{b = 1}^k \KL(p^*_{B_b}\|q_{B_b}).
    \label{eq:KL-block-decomposition}
  \end{align}
  Applying this identity to $q = p_L$ gives \eqref{eq:KL-block-decomposition-DPP}.
\end{proof}

\begin{proposition}
  \label{prop:error-matching-ub}
  The DPP approximation error satisfies
  \begin{equation}
    \KL(p^*\|\DPP) \leq \KL(p^*\|\nu_I) - \max_{\mathcal{M}:\ p^* \text{-repulsive matching}} \sum_{(i,j) \in \mathcal{M}} I_{p^*}(X_i;X_j).
  \end{equation}
  Here a $p^*$-repulsive matching is a collection $\mathcal{M}$ of disjoint pairs such that, for every $(i,j) \in \mathcal{M}$,
  \begin{equation}
    \eta^*_{\{i,j\}} \leq \eta^*_{\{i\}}\eta^*_{\{j\}}.
  \end{equation}
\end{proposition}
\begin{proof}
  Fix a $p^*$-repulsive matching $\mathcal{M} = \{(i_1,j_1), \dots, (i_k,j_k)\}$. Consider an $L$-ensemble kernel that is block diagonal along this matching: it has $2 \times 2$ blocks $L_{\{i_b,j_b\}}$ for $b=1, \dots, k$, and one-dimensional diagonal blocks $L_{ii}$ for items not contained in any matched pair. By Lemma~\ref{lem:block-diagonal-DPP},
  \begin{align}
    \KL(p^*\|p_L) &= \KL(p^*\|\widetilde{p}^{\,*}) + \sum_{b = 1}^k \KL(p^*_{\{i_b,j_b\}}\|p_{L_{\{i_b,j_b\}}})\\
    &\quad + \sum_{i \in [N] \setminus \bigcup_{b = 1}^k \{i_b,j_b\}} \KL(p^*_{\{i\}}\|p_{L_{\{i\}}}).
    \label{eq:KL-matching-decomposition}
  \end{align}
  By Lemma~\ref{lem:DPP-exact-representation}, the second and third terms can be made zero by choosing the block kernels appropriately. Hence
  \begin{equation}
    \KL(p^*\|\DPP) \leq \KL(p^*\|\widetilde{p}^{\,*}).
    \label{eq:error-matching-ub}
  \end{equation}

  On the other hand, the independent distribution $\nu_I$ also factorizes according to the same matching. Applying \eqref{eq:KL-block-decomposition} to $\nu_I$ gives
  \begin{align}
    \KL(p^*\|\nu_I) &= \KL(p^*\|\widetilde{p}^{\,*}) + \sum_{b = 1}^k \KL(p^*_{\{i_b,j_b\}}\|p^*_{\{i_b\}} \otimes p^*_{\{j_b\}})\\
    &= \KL(p^*\|\widetilde{p}^{\,*}) + \sum_{b = 1}^k I_{p^*}(X_{i_b}; X_{j_b}).
  \end{align}
  Combining this identity with \eqref{eq:error-matching-ub}, and then maximizing over all $p^*$-repulsive matchings, proves the claim.
\end{proof}

\subsection{Off-Block Perturbations and Local Optimality}

Proposition~\ref{prop:error-matching-ub} shows that a block-diagonal kernel based on a repulsive matching can improve the approximation error over the independent distribution. However, this construction only uses pairwise interactions. The following result gives a local criterion for whether a block-structured diversity matrix can be further improved by introducing off-block interactions.

\begin{proposition}
  \label{prop:characterization-improvement}
  Let $B_1, \dots, B_k$ be a partition of $[N]$. Suppose that a diversity matrix $D$ is block diagonal with respect to this partition and is a local minimizer of $\KL(p^*\|\nu_D)$ among diversity matrices with the same block-diagonal structure. We call a symmetric matrix $H$ an off-block perturbation if $H_{B_b} = O_{B_b}$ for all $b = 1, \dots, k$. Then, using $\Delta_{2,D}$ from \eqref{eq:perturbation-coeff}, the following statements hold.
  \begin{enumerate}
    \item If $\Delta_{2,D}(H) > 0$ for every nonzero off-block perturbation $H$, then $D$ is also a local minimizer among all diversity matrices.
    \item If there exists a nonzero off-block perturbation $H$ such that $\Delta_{2,D}(H) < 0$, then for all sufficiently small $t > 0$,
    \begin{equation}
      \KL(p^*\|\nu_{D+tH}) < \KL(p^*\|\nu_D).
    \end{equation}
  \end{enumerate}
\end{proposition}

\begin{proof}
  Let $\mathcal{L}(D) \coloneqq \KL(p^*\|\nu_D)$. For an arbitrary symmetric perturbation $E$ satisfying $E_{ii} = 0$ and $\|E\|_F = 1$, decompose it into its block and off-block components as
  \begin{equation}
    E = G + H.
  \end{equation}
  A direct calculation gives
  \begin{align}
    \Delta_{1,D}(E) &= \Delta_{1,D}(G) + \Delta_{1,D}(H),\\
    \Delta_{2,D}(E) &= \Delta_{2,D}(G) + \Delta_{2,D}(H).
  \end{align}
  To verify the decomposition of $\Delta_{2,D}(E)$, observe that
  \begin{align}
      \tr\{(D_J^{-1} E_J)^2\} &= \tr\{(D_J^{-1} G_J + D_J^{-1} H_J)^2\}\\
      &= \tr\{(D_J^{-1} G_J)^2\} + \tr\{(D_J^{-1} H_J)^2\} + 2\tr\{D_J^{-1} G_J D_J^{-1} H_J\}.
  \end{align}
  Since $D_J^{-1} G_J$ is block diagonal and $D_J^{-1} H_J$ is off-block, their product is off-block and therefore has zero trace. For the same reason, $\tr\{D_J^{-1} H_J\}=0$ for every $J \subset [N]$, and hence $\Delta_{1,D}(H) = 0$.

  Uniformly over the admissible perturbation directions, \eqref{eq:KL-perturbation} and Lemma~\ref{lem:KL-perturbation-order} imply, for all sufficiently small $t > 0$, that
  \begin{equation}
    \mathcal{L}(D + tE) - \mathcal{L}(D) = -t\Delta_{1,D}(G) + \frac{t^2}{2} \{f(G) + \Delta_{2,D}(G) + \Delta_{2,D}(H)\} + O(t^3),
  \end{equation}
  where the nonnegative term $f(G)$ arises from $\KL(\nu_D\|\nu_{D+tE})$. Applying this expansion with $H = O$ and using the local optimality of $D$ within the block-diagonal family, we obtain
  \begin{equation}
    \Delta_{1,D}(G) = 0, \quad f(G) + \Delta_{2,D}(G) \geq 0.
  \end{equation}

  If there exists an off-block perturbation $H \neq O$ with $\Delta_{2,D}(H) < 0$, take $E = H/\|H\|_F$. Then
  \begin{equation}
    \mathcal{L}(D + tE) - \mathcal{L}(D) = \frac{t^2}{2\|H\|_F^2} \Delta_{2,D}(H) + O(t^3),
  \end{equation}
  which is strictly negative for all sufficiently small $t > 0$. Replacing $t$ by $t\|H\|_F$ proves the claimed inequality for $D + tH$ and hence the second assertion.

  Conversely, suppose that $\Delta_{2,D}(H) > 0$ for every nonzero off-block $H$. By compactness of the unit sphere in the finite-dimensional off-block subspace, there exists
  \begin{equation}
    \lambda \coloneqq \min_{H: \text{off-block}, \medspace \|H\|_F=1} \Delta_{2,D}(H) > 0.
  \end{equation}

  Since $D+tG$ remains block diagonal, the same trace argument as above gives $\Delta_{1,D+tG}(H)=0$. Therefore, by \eqref{eq:KL-perturbation} and $\KL(\nu_{D + tG}\|\nu_{D + tE}) \geq 0$, there exists $\delta > 0$, independent of $E$, such that, for every admissible $E$ and every $0 < t < \delta$,
  \begin{equation}
    \mathcal{L}(D + tG + tH) - \mathcal{L}(D + tG) \geq \frac{t^2}{2} \Delta_{2, D + tG}(H) + \sum_{k \geq 3} \frac{(-1)^k}{k} t^k \Delta_{k,D + tG}(H).
    \label{eq:KL-offblock-perturbation}
  \end{equation}
  For the convergence requirement in \eqref{eq:KL-perturbation}, we may initially take $\delta = \lambda_{\min}(D)/2$. Indeed, for every nonempty $J \subset [N]$,
  \begin{equation}
    \|(D_J + tG_J)^{-1}H_J\|_F \leq \|(D_J + tG_J)^{-1}\|_\mathrm{op}\|H_J\|_F \leq \frac{1}{\lambda_{\min}(D_J + tG_J)},
  \end{equation}
  where we used $\|H_J\|_F \leq \|H\|_F \leq \|E\|_F = 1$. Moreover,
  \begin{equation}
    \lambda_{\min}(D_J + tG_J) \geq \lambda_{\min}(D_J) - t\|G_J\|_\mathrm{op} \geq \lambda_{\min}(D) - t.
    \label{eq:D_J+tG_J_eigen_ineq}
  \end{equation}
  Hence, for every $0 < t < \delta$,
  \begin{equation}
    t\|(D_J + tG_J)^{-1}H_J\|_F \leq \frac{t}{\lambda_{\min}(D) - t} < 1.
  \end{equation}
  Since $\|G\|_F \leq 1$, $D+tG$ converges to $D$ uniformly over the admissible directions $E$ as $t \downarrow 0$. Thus, by continuity of $(D,H) \mapsto \Delta_{2,D}(H)$ and compactness of the unit sphere in the finite-dimensional off-block subspace, we may decrease $\delta$, independently of $E$, so that, for every admissible $E$ with $H \neq O$ and every $0 < t < \delta$,
  \begin{equation}
    \Delta_{2,D + tG}\left(\frac{H}{\|H\|_F}\right) \geq \frac{\lambda}{2}.
  \end{equation}
  Consequently, for every off-block $H$ and every $0 < t < \delta$,
  \begin{equation}
    \Delta_{2,D + tG}(H) \geq \frac{\lambda}{2}\|H\|_F^2.
  \end{equation}
  For the series in the second term of \eqref{eq:KL-offblock-perturbation}, Lemma~\ref{lem:perturbation-coeff-ub} gives, for every nonzero off-block $H$ and every $0 < t < \delta$,
  \begin{align}
    \left|\Delta_{k,D + tG}\left(\frac{H}{\|H\|_F}\right)\right| &\leq 2\left(\max_{\emptyset \neq J \subset [N]} \|(D_J + tG_J)^{-1}\|_\mathrm{op}\right)^k\\
    &\leq \frac{2}{(\lambda_{\min}(D) - t)^k} \leq 2\left(\frac{2}{\lambda_{\min}(D)}\right)^k.
  \end{align}
  The second inequality follows from \eqref{eq:D_J+tG_J_eigen_ineq}. Therefore, setting $M \coloneqq 2/\lambda_{\min}(D)$ and using the homogeneity of $\Delta_{k,D+tG}$ in $H$ gives $|\Delta_{k,D + tG}(H)| \leq 2M^k\|H\|_F^k$. Since $\|H\|_F \leq \|E\|_F = 1$, for $0 < t < \min\{\delta, 1/(2M)\}$,
  \begin{equation}
    \left|\sum_{k \geq 3} \frac{(-1)^k}{k} t^k \Delta_{k,D + tG}(H)\right| \leq 2 \sum_{k \geq 3} \frac{(M\|H\|_F t)^k}{k} \leq \frac{4}{3} \|H\|_F^2 M^3 t^3.
  \end{equation}
  Therefore, the right-hand side of \eqref{eq:KL-offblock-perturbation} is bounded below by
  \begin{equation}
    \|H\|_F^2 t^2 \left(\frac{\lambda}{4} - \frac{4}{3} M^3 t\right),
  \end{equation}
  which is nonnegative whenever
  \begin{equation}
    0 < t < \min\left\{\delta, \frac{1}{2M}, \frac{3\lambda}{16M^3}\right\}.
  \end{equation}
  Hence, for all sufficiently small $t > 0$,
  \begin{equation}
    \mathcal{L}(D + tG + tH) \geq \mathcal{L}(D + tG).
  \end{equation}
  Since $D$ is locally optimal within the block-diagonal family and $\|G\|_F \leq 1$, we also have $\mathcal{L}(D + tG) \geq \mathcal{L}(D)$ uniformly over the admissible directions for all sufficiently small $t > 0$. Thus, $D$ is a local minimizer among all diversity matrices.
\end{proof}

Proposition~\ref{prop:characterization-improvement} indicates that, once a diversity matrix has been locally optimized within a given block structure, off-block directions are the natural directions to examine for further local improvement.

\section{Numerical Illustration for \texorpdfstring{$N = 3$}{N = 3}}
\label{sec:numerical-example}

We consider the case $N = 3$ and the following two-parameter family of distributions:
\begin{align}
  p_{\alpha,\beta}(A) &= \exp\left[\alpha 1\{|A| = 2\} + \beta 1\{|A| = 3\} - \psi(\alpha,\beta)\right], \label{eq:curie-weiss}\\
  \psi(\alpha,\beta) &= \log(4 + 3e^\alpha + e^\beta).
\end{align}
This is a subfamily of the log-linear model \eqref{eq:log-linear} in which the interaction parameters are constant within each subset size:
\begin{equation}
  \theta^{\{i\}} = 0, \quad \theta^{\{i,j\}} = \alpha, \quad \theta^{\{1,2,3\}} = \beta - 3\alpha.
\end{equation}
For each choice of $(\alpha, \beta)$, we regard $p_{\alpha,\beta}$ as the target distribution $p^*$ and numerically approximate
\begin{equation}
  \Delta \coloneqq \KL(p_{\alpha,\beta}\|\nu_I) - \KL(p_{\alpha,\beta}\|\DPP).
\end{equation}
The resulting values are shown in Figure~\ref{fig:heatmap}. For each parameter value, we also compute certified upper and lower bounds, denoted by $\Delta_{\mathrm{upper}}$ and $\Delta_{\mathrm{lower}}$, respectively, on $\Delta$ using a branch-and-bound procedure. The computational details are given in Appendix~\ref{sec:numerical-bound-computation}. Points satisfying $\Delta_{\mathrm{lower}} > \delta$, which certify a non-negligible improvement over the independent distribution $\nu_I$, are marked by black dots. Points satisfying $\Delta_{\mathrm{upper}} < \delta$, which certify that no meaningful improvement over $\nu_I$ is possible, are marked by crosses. We use $\delta = 10^{-8}$.

\begin{figure}[!t]
  \centering
  \includegraphics[width=\linewidth]{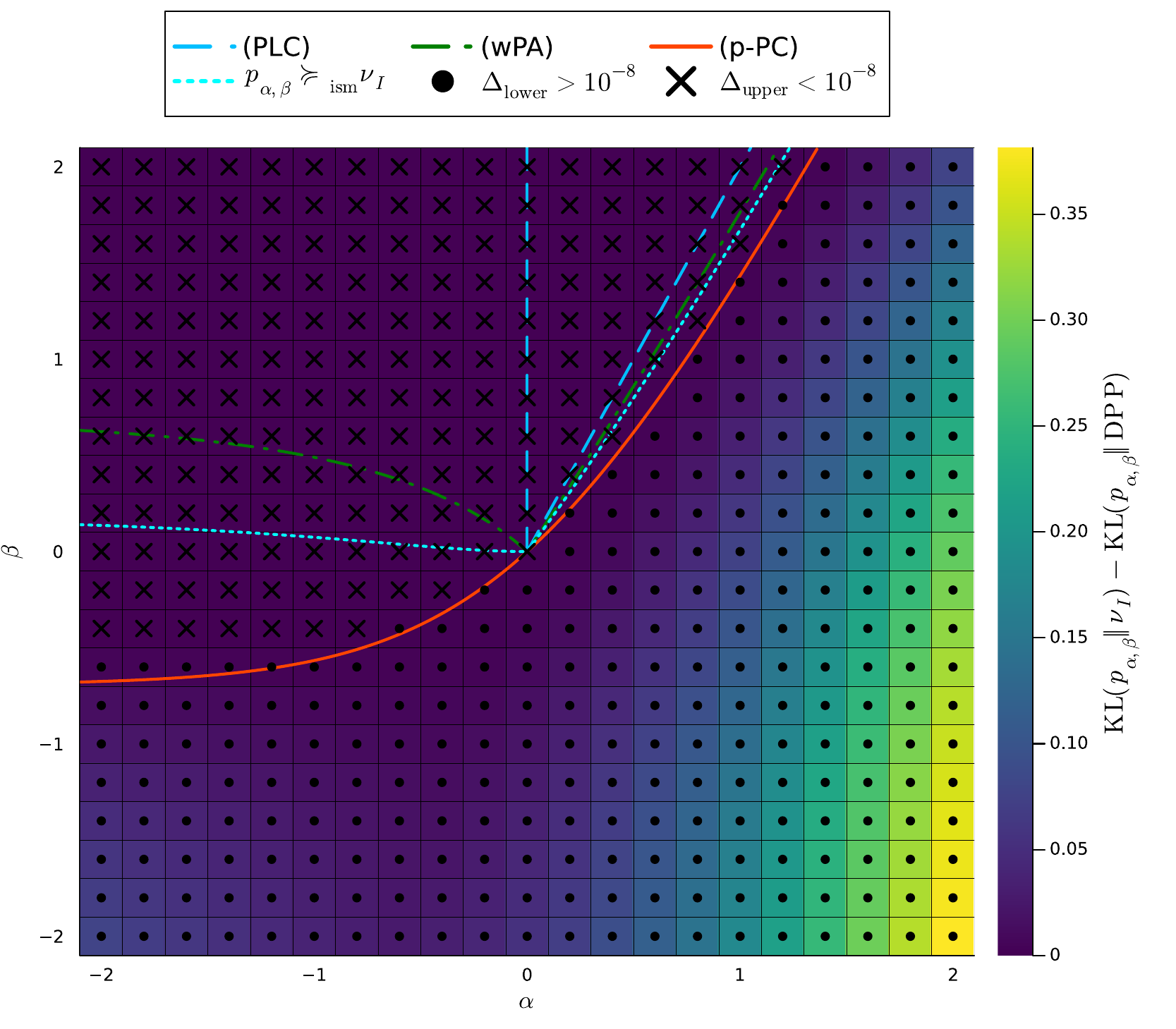}
  \caption{Numerical approximations of $\KL(p_{\alpha,\beta}\|\nu_I) - \KL(p_{\alpha,\beta}\|\DPP)$ for the target family $p_{\alpha,\beta}$ defined in \eqref{eq:curie-weiss}, computed over a grid of parameter values $(\alpha, \beta)$. Each cell uses the value of $(\alpha, \beta)$ at its center. The grid ranges over $-2 \leq \alpha \leq 2$ and $-2 \leq \beta \leq 2$ with mesh size $0.2$}
  \label{fig:heatmap}
\end{figure}

The dependence structure of this family changes with $(\alpha, \beta)$. Since $N = 3$, the sample space $2^{[N]}$ is small. Moreover, both $p_{\alpha,\beta}$ and $\nu_I$ are exchangeable, in the sense that $p_{\alpha,\beta}(A) = p_{\alpha,\beta}(\sigma(A))$ for every permutation $\sigma:[N] \to [N]$, where $\sigma(A) = \{\sigma(i) \mid i \in A\}$. These facts allow the dependence properties to be expressed explicitly as constraints on $(\alpha, \beta)$. The resulting conditions are summarized in Table~\ref{tb:cond-p-dependence}, and their derivations are given in Appendix~\ref{sec:dependence-conditions-n3}.

\begin{table}[!tbp]
  \centering
  \caption{Necessary and sufficient conditions, in terms of $(\alpha, \beta)$, for several attractive-dependence properties.}
  \begin{tabular}{l|l}
    (PLC) & $0 \leq \alpha$, $\beta \geq 2 \alpha$ \\
    (wPA) & $\beta \geq \begin{cases}
      \log(2 - e^\alpha), & \alpha \leq 0, \\
      \log\left(\frac{2e^{2\alpha} + e^\alpha}{3}\right), & \alpha \geq 0
    \end{cases}$ \\
    (p-PC) & $\beta \geq \log\left(\frac{e^{2\alpha} + 1}{2}\right)$ \\
    $p_{\alpha,\beta} \succism \nu_I$ & $\beta \geq \begin{cases}
      \log\left((e^\alpha + 3)^{3/2} - 3e^\alpha - 4\right), & \alpha \leq 0,\\
      \log\left(\frac{3e^{2\alpha} - 12e^\alpha - 13 + \sqrt{(e^\alpha + 3)^3 (9e^\alpha + 7)}}{10}\right), & \alpha \geq 0
    \end{cases}$
  \end{tabular}
  \label{tb:cond-p-dependence}
\end{table}

Based on these characterizations, Figure~\ref{fig:heatmap} also displays the boundary curves for the regions satisfying PLC, wPA, p-PC, and $p_{\alpha,\beta} \succism \nu_I$. For each curve, the region above the curve, corresponding to larger values of $\beta$, is the region where the corresponding property holds. The figure shows that, throughout the region $p_{\alpha,\beta} \succism \nu_I$, there is no substantial improvement over the independent approximation $\nu_I$. This observation is consistent with Theorem~\ref{thm:nonvalidity-ism}.

On the other hand, there are parameter values, such as $(\alpha, \beta) = (-2, -0.6)$, for which p-PC holds while $p_{\alpha,\beta} \succism \nu_I$ does not, and a DPP approximation improves over $\nu_I$. This illustrates the point discussed in Section~\ref{sec:DPP-error-lb-for-pPC}: p-PC is a weaker attractive-dependence condition and does not imply the assumption of Theorem~\ref{thm:nonvalidity-ism}.

\section{Conclusion}
\label{sec:conclusion}

We studied the population-level approximation of strictly positive distributions on random subsets by determinantal point processes under the forward Kullback--Leibler divergence. Our central result is a structural reduction based on the quality--diversity decomposition: the optimization separates into matching the first-order inclusion probabilities of the target distribution and optimizing the diversity component responsible for the interactions represented by the DPP. This reduction provides a common basis for the two dependence regimes considered in this paper. For attractively dependent targets, we obtained conditions under which the independent product distribution with the same first-order marginals is globally optimal. For repulsively dependent targets, we constructed approximations that improve upon this independent baseline and quantified the guaranteed improvement using mutual information associated with negatively correlated pairs.

The strict positivity assumption on the target distribution $p^*$ should be noted. This assumption is mainly technical: it ensures that $p^*$ lies in the interior of the probability simplex, where the global log-linear coordinates used in our analysis are well defined. If distributions on the boundary were allowed, some of the natural parameters could cease to be finite, and the present information-geometric argument would require modification. We have also focused on population-level approximation and have not addressed estimation from finitely many observations. Since the minimization of the forward KL divergence considered here is the population counterpart of maximum likelihood estimation, the present analysis provides a natural starting point for studying consistency, finite-sample error, and the estimation of the optimal diversity structure from data.

An important open problem is to sharpen, and where possible globalize, the local analyses developed in this paper. On the attractive side, we obtained a global optimality result for target distributions satisfying $p^* \succism \nu_I$, whereas pairwise positive correlation yielded only local optimality conditions and approximation-error lower bounds based on individual pairs. A corresponding gap remains on the repulsive side. The matching construction provides a global improvement guarantee, but it exploits only disjoint pairwise interactions and is not expected to be sharp in general. A more complete theory should incorporate interactions among three or more elements and characterize both the globally optimal diversity structure and the maximal improvement over the independent approximation. The off-block perturbation analysis indicates one possible route toward exploiting such higher-order structure, but it currently provides only local criteria. Extending these criteria to global and quantitatively sharp results is therefore one of the main directions for future work.

\begin{appendices}
\renewcommand{\theHequation}{\Alph{section}.\arabic{equation}}
\section{Technical Proofs}
\subsection{Proof of Mixed-coordinate Pythagorean decomposition}
\label{sec:proof-k-cut-pythagoras}
It remains to justify the existence assertion used in the proof of the mixed-coordinate Pythagorean decomposition in Section~\ref{sec:ig-coordinates}: namely, for given points $p,r$, there exists a point $q$ whose $\eta$-coordinates up to order $k$ coincide with those of $p$ and whose $\theta$-coordinates of order $k+1$ and higher coincide with those of $r$.

\begin{proof}
  Consider the subfamily of the log-linear model \eqref{eq:log-linear} in which the $\theta$-coordinates of order $k+1$ and higher are fixed at $\theta_{(k+1)}(r), \dots, \theta_{(N)}(r)$:
  \begin{equation}
    q_{\theta_{(1:k)}, \theta_{(k+1:N)}(r)}(A) \coloneqq \exp\left[c(A) + \sum_{\substack{\emptyset \neq I \subset [N]\\ |I| \leq k}} \theta^I 1\{I \subset A\} - \psi_{\theta_{(k+1:N)}}(\theta_{(1:k)})\right].
  \end{equation}
  Here $\theta_{(1:k)} = (\theta_{(1)}, \dots, \theta_{(k)})$, and $\theta_{(k+1:N)}$ is defined analogously. The offset term is
  \begin{equation}
    c(A) = \sum_{\substack{I \subset [N] \\ |I| > k}} \theta^I(r) 1\{I \subset A\}.
  \end{equation}
  This is a regular and minimal exponential family with sufficient-statistic vectors $(1\{I \subset Y\})_{\emptyset \neq I \subset [N], \medspace |I| \leq k}$. Its mean-parameter space is the convex hull of these vectors. Moreover, $\eta_{(1:k)}(p)$ is their convex combination with positive weights, and hence it lies in the interior of the mean-parameter space. Therefore, by \cite[Proposition~3.2 and Theorem~3.3]{wainwright-2008-graphical}, the gradient map
  \begin{equation}
    \nabla_{\theta_{(1:k)}}\psi_{\theta_{(k+1:N)}}: \theta_{(1:k)} \mapsto \nabla_{\theta_{(1:k)}} \psi_{\theta_{(k+1:N)}}(\theta_{(1:k)}) = \eta_{(1:k)}
  \end{equation}
  is a bijection from the natural-parameter space onto the interior of the mean-parameter space. Consequently, there exists a value of $\theta_{(1:k)}$ corresponding to $\eta_{(1:k)}(p)$. The distribution $q_{\theta_{(1:k)}, \theta_{(k+1:N)}(r)}$ with this value of $\theta_{(1:k)}$ is the desired point $q$.
\end{proof}

\subsection{A Discrete Proof of the Supermodular-Order Comparison}
\label{sec:discrete-wPA-sm-ineq}

\begin{proof}
  The argument follows the same strategy as the proof in \cite{christofides-2004-connection}, specialized here to binary random vectors. The setting is the same as in the proof of Corollary~\ref{cor:nonvalidity-p-dependence}. Let $X \in \{0,1\}^N$ satisfy weak positive association in the sense of \eqref{eq:vector-wPA}, and let $Z \in \{0,1\}^N$ be the vector obtained by replacing the components of $X$ with independent copies having the same one-dimensional marginal distributions. We may assume that $X$ and $Z$ are independent.

  For $i=0,1,\dots,N$, define the interpolating random vector
  \begin{equation}
    W^{(i)} \coloneqq (Z_1, \dots, Z_i, X_{i+1}, \dots, X_N).
  \end{equation}
  Then $W^{(0)} = X$ and $W^{(N)} = Z$. It suffices to show that, for every supermodular function $h:\{0,1\}^N \to \R$,
  \begin{equation}
    \E[h(W^{(i)})] \geq \E[h(W^{(i+1)})] \quad (i=0, 1, \dots, N-1).
  \end{equation}
  Chaining these inequalities yields $\E[h(X)] \geq \E[h(Z)]$.

  Fix $i$ and condition on the first $i$ coordinates. For $u \in \{0,1\}^i$, define
  \begin{equation}
    h_u:\{0,1\}^{N-i} \to \R, \quad h_u(x_{i+1}, \dots, x_N) = h(u_1, \dots, u_i, x_{i+1}, \dots, x_N).
  \end{equation}
  Writing $x_{i+1:N} = (x_{i+1}, \dots, x_N)$, we have
  \begin{equation}
    h_u(x_{i+1:N}) = h_u(0,x_{i+2:N}) + x_{i+1} \Delta_{i+1} h_u(x_{i+2:N}),
    \label{eq:hu-representation}
  \end{equation}
  where
  \begin{equation}
    \Delta_{i+1} h_u(x_{i+2:N}) \coloneqq h_u(1, x_{i+2:N}) - h_u(0, x_{i+2:N}).
  \end{equation}

  The function $\Delta_{i+1} h_u$ is increasing in $x_{i+2:N}$. Indeed, if $x_{i+2:N} \preceq x'_{i+2:N}$, set
  \begin{equation}
    v_1=(u, 0, x'_{i+2:N}), \quad v_2=(u, 1, x_{i+2:N}).
  \end{equation}
  Then
  \begin{align}
    &\Delta_{i+1} h_u(x'_{i+2:N}) - \Delta_{i+1} h_u(x_{i+2:N})\\
    &= h(v_1 \lor v_2) - h(v_1) - h(v_2) + h(v_1 \land v_2) \geq 0,
  \end{align}
  where the last inequality follows from the supermodularity condition \eqref{eq:vector-supermodular}.

  Using \eqref{eq:hu-representation},
  \begin{align}
    &\E[h_u(X_{i+1}, X_{i+2:N})] - \E[h_u(Z_{i+1}, X_{i+2:N})]\\
    &= \E[X_{i+1} \Delta_{i+1} h_u(X_{i+2:N})] - \E[Z_{i+1}] \E[\Delta_{i+1} h_u(X_{i+2:N})]\\
    &= \Cov(X_{i+1}, \Delta_{i+1} h_u(X_{i+2:N})).
    \label{eq:hu-diff-cov}
  \end{align}
  The covariance on the right-hand side is nonnegative by weak positive association, because both functions are increasing and depend on disjoint coordinate sets. Therefore,
  \begin{equation}
    \E[h_u(X_{i+1}, X_{i+2:N})] \geq \E[h_u(Z_{i+1}, X_{i+2:N})].
  \end{equation}
  Taking expectation also over $u = Z_{1:i}$ gives
  \begin{equation}
    \E[h(W^{(i)})] \geq \E[h(W^{(i+1)})].
  \end{equation}
  This proves $X \succsm Z$. If $X$ is negatively associated, the covariance in \eqref{eq:hu-diff-cov} is nonpositive, and the inequalities are reversed.
\end{proof}

\subsection{Perturbation Analysis of the DPP Approximation Error}
\label{sec:derivation-KL-perturbation}

Applying the extended Pythagorean theorem to the three points $p^*$, $\nu_D$, and $\nu_{D_t}$ gives
\begin{equation}
  \KL(p^*\|\nu_{D_t}) = \KL(p^*\|\nu_D) + \KL(\nu_D\|\nu_{D_t}) + \ang*{\eta^* - \eta(\nu_D), \theta(\nu_D) - \theta(\nu_{D_t})}.
  \label{eq:KL-perturbation-pythagoras}
\end{equation}
Using $\eta^*_{(1)} = \eta_{(1)}(\nu_D)$, the inner-product term can be rewritten, in the same way as in the derivation of \eqref{eq:inner-prod-logdet}, as
\begin{equation}
  \E_{p^*}[\log\det D_Y - \log\det(D_t)_Y] - \E_{\nu_D}[\log\det D_Y - \log\det(D_t)_Y].
  \label{eq:KL-perturbation-iprod}
\end{equation}
For $t$ satisfying $t\|D_J^{-1}E_J\|_F < 1$ for every nonempty $J \subset [N]$, we have
\begin{align}
  \log\det(D_t)_J &= \log\det D_J + \tr\log(I + tD_J^{-1}E_J)\\
  &= \log\det D_J - \sum_{k \geq 1} \frac{(-1)^k}{k} t^k \tr\{(D_J^{-1}E_J)^k\}.
  \label{eq:logdet-taylor-expansion}
\end{align}
Combining \eqref{eq:KL-perturbation-pythagoras}, \eqref{eq:KL-perturbation-iprod}, and \eqref{eq:logdet-taylor-expansion} yields
\begin{align}
  \KL(p^*\|\nu_{D_t}) = &\KL(p^*\|\nu_D) + \KL(\nu_D\|\nu_{D_t})\\
  &+ \sum_{k \geq 1} \frac{(-1)^k}{k} t^k \left(\E_{p^*}[\tr\{(D_Y^{-1}E_Y)^k\}] - \E_{\nu_D}[\tr\{(D_Y^{-1}E_Y)^k\}]\right).
\end{align}
By the definition of $\Delta_{k,D}(E)$ in \eqref{eq:perturbation-coeff}, this is precisely \eqref{eq:KL-perturbation}.

We next control the second term in this expansion.
\begin{lemma}
  \label{lem:KL-perturbation-order}
  Let $D_t = D + tE$, where $E$ is symmetric with $E_{ii} = 0$, and let $t$ be sufficiently small that $D_t \succ O$. Assume that $E = G + H$ for some symmetric matrices $G$ and $H$ satisfying $\tr(D_J^{-1}H_J) = 0$ for all nonempty $J \subset [N]$. Then, in \eqref{eq:KL-perturbation}, there exists a nonnegative function $f(G) \geq 0$, depending only on $G$, such that
  \begin{equation}
    \KL(\nu_D\|\nu_{D_t}) = \frac{t^2}{2} f(G) + O(t^3).
  \end{equation}
  In particular, if $G = O$, so that $E = H$, then $f(G) = 0$ and
  \begin{equation}
    \KL(\nu_D\|\nu_{D_t}) = O(t^4).
  \end{equation}
\end{lemma}
\begin{proof}
  For $J \subset [N]$ with $|J| = 2$, \eqref{eq:dpp-theta-coordinate} and \eqref{eq:logdet-taylor-expansion} give
  \begin{equation}
    \left.\frac{d}{dt} \theta^J(\nu_{D_t})\right|_{t = 0} = \tr(D_J^{-1}(G_J + H_J)) = \tr(D_J^{-1}G_J),
  \end{equation}
  which depends only on $G$. Suppose inductively that the same conclusion holds for all subsets of size between $2$ and $k$. For $|J| = k+1$, \eqref{eq:dpp-theta-coordinate} gives
  \begin{equation}
    \left.\frac{d}{dt} \theta^J(\nu_{D_t})\right|_{t = 0} = \tr(D_J^{-1}G_J) - \sum_{J' \subsetneq J; \medspace |J'| \geq 2} \left.\frac{d}{dt} \theta^{J'}(\nu_{D_t})\right|_{t = 0},
  \end{equation}
  and the right-hand side again depends only on $G$. Hence, for all $|J| \geq 2$, the first derivative of $\theta^J(\nu_{D_t})$ at $t = 0$ depends only on $G$. Thus, for some vector $v_G$,
  \begin{equation}
    \theta_{(2+)}(\nu_{D_t}) = \theta_{(2+)}(\nu_D) + t v_G + O(t^2).
  \end{equation}

  In the $1$-cut mixed coordinates, this means
  \begin{equation}
    \delta \coloneqq \omega_1(\nu_{D_t}) - \omega_1(\nu_D) = (0; t v_G + O(t^2)).
  \end{equation}
  Let $g_{\theta_{(2+)}\theta_{(2+)}}(\omega_1(\nu_D))$ denote the principal block of the Fisher information matrix $g(\omega_1(\nu_D))$ corresponding to the $\theta_{(2+)}$ components. The Taylor expansion of KL divergence in mixed coordinates gives
  \begin{align}
    \KL[\omega_1(\nu_D) : \omega_1(\nu_{D_t})] &= \frac{1}{2} \delta^\top g(\omega_1(\nu_D)) \delta + O(\|\delta\|^3)\\
    &= \frac{t^2}{2} v_G^\top g_{\theta_{(2+)}\theta_{(2+)}}(\omega_1(\nu_D)) v_G + O(t^3),
    \label{eq:KL-mixed-taylor-expansion}
  \end{align}
  Since $g(\omega_1(\nu_D))$ is positive definite, so is its principal block $g_{\theta_{(2+)}\theta_{(2+)}}(\omega_1(\nu_D))$. Setting
  \begin{equation}
    f(G) = v_G^\top g_{\theta_{(2+)}\theta_{(2+)}}(\omega_1(\nu_D)) v_G \geq 0
  \end{equation}
  proves the first claim. If $G = O$, then $v_G = 0$, so $\|\delta\| = O(t^2)$ and \eqref{eq:KL-mixed-taylor-expansion} gives $\KL(\nu_D\|\nu_{D_t}) = O(t^4)$.
\end{proof}

We finally bound the coefficients $\Delta_{k,D}(E)$.

\begin{lemma}
  \label{lem:trace-frobenius-ineq}
  Let $E \in \R^{n \times n}$ be symmetric. Then, for any integer $k \geq 2$,
  \begin{equation}
    |\tr E^k| \leq \|E\|_\mathrm{op}^{k-2} \|E\|_F^2.
  \end{equation}
\end{lemma}
\begin{proof}
  Let $E = U \Lambda U^\top$ be a spectral decomposition, with $\Lambda = \Diag(\lambda_1, \dots, \lambda_n)$. Then
  \begin{equation}
    \tr E^k = \sum_{i=1}^{n} \lambda_i^k,
  \end{equation}
  and hence
  \begin{equation}
    |\tr E^k| \leq \sum_{i=1}^{n}|\lambda_i|^k \leq \left(\max_{1 \leq i \leq n} |\lambda_i|\right)^{k-2} \sum_{i=1}^{n} |\lambda_i|^2 = \|E\|_\mathrm{op}^{k-2} \|E\|_F^2.
  \end{equation}
\end{proof}

\begin{lemma}
  \label{lem:perturbation-coeff-ub}
  Let $D \succ O$ be a diversity matrix. Then, for every symmetric matrix $E$ satisfying $\|E\|_F = 1$ and every integer $k \geq 2$,
  \begin{equation}
    |\Delta_{k,D}(E)| \leq 2\left(\max_{\emptyset \neq J \subset[N]} \|D_J^{-1}\|_\mathrm{op}\right)^k.
  \end{equation}
\end{lemma}
\begin{proof}
  For every nonempty $J \subset [N]$,
  \begin{equation}
    \tr\{(D_J^{-1}E_J)^k\} = \tr\{(D_J^{-1/2} E_J D_J^{-1/2})^k\}.
  \end{equation}
  Setting $A_J = D_J^{-1/2} E_J D_J^{-1/2}$, Lemma~\ref{lem:trace-frobenius-ineq} yields
  \begin{equation}
    |\tr\{(D_J^{-1} E_J)^k\}| \leq \|A_J\|_\mathrm{op}^{k-2} \|A_J\|_F^2.
  \end{equation}
  Moreover,
  \begin{equation}
      \|A_J\|_\mathrm{op} \leq \|D_J^{-1}\|_\mathrm{op} \|E_J\|_\mathrm{op} \leq \|D_J^{-1}\|_\mathrm{op},
  \end{equation}
  and
  \begin{equation}
      \|A_J\|_F \leq \|D_J^{-1}\|_\mathrm{op} \|E_J\|_F \leq \|D_J^{-1}\|_\mathrm{op}.
  \end{equation}
  Therefore,
  \begin{equation}
      |\tr\{(D_J^{-1} E_J)^k\}| \leq \|D_J^{-1}\|_\mathrm{op}^k.
  \end{equation}
  It follows that
  \begin{align}
    |\Delta_{k,D}(E)| &\leq \sum_{\emptyset \neq J \subset [N]} |p^*(J) - \nu_D(J)| |\tr\{(D_J^{-1} E_J)^k\}|\\
    &\leq \|p^* - \nu_D\|_1 \max_{\emptyset \neq J \subset[N]} \|D_J^{-1}\|_\mathrm{op}^{k}\\
    &\leq 2\max_{\emptyset \neq  J \subset [N]} \|D_J^{-1}\|_\mathrm{op}^{k}.
  \end{align}
\end{proof}

\section{Details for the Three-Element Model}
\subsection{Characterization of the Dependence Conditions}
\label{sec:dependence-conditions-n3}

Among the characterizations of the four dependence properties listed in Table~\ref{tb:cond-p-dependence}, those for the positive lattice condition~\eqref{eq:PLC} and pairwise positive correlation~\eqref{eq:p-PC} follow by direct calculation. We derive the remaining two characterizations below. Throughout this subsection, we use the binary random-vector notation introduced in Section~\ref{sec:global-optimality}.

We first give the following characterization of weak positive association in terms of binary-valued increasing functions. It is obtained by restricting the argument for positive association in \cite[Theorem~3.1]{esary-1967-association} to disjoint coordinate sets; the proof is essentially the same.
\begin{lemma}
  \label{lem:wPA-condition}
  The random vector $X$ is weakly positively associated in the sense of \eqref{eq:vector-wPA} if and only if, for every pair of disjoint sets $I,J \subset [N]$ and every pair of binary-valued increasing functions $\gamma$ and $\delta$ of $(x_i)_{i \in I}$ and $(x_j)_{j \in J}$, respectively,
  \begin{equation}
    \Cov(\gamma((X_i)_{i \in I}), \delta((X_j)_{j \in J})) \geq 0,
    \label{eq:wPA-condition-assumption}
  \end{equation}
  holds.
\end{lemma}

\begin{proof}
  Necessity follows immediately from the definition~\eqref{eq:vector-wPA}. For sufficiency, let $I,J \subset [N]$ be disjoint, and let $f$ and $g$ be arbitrary increasing functions of $(x_i)_{i \in I}$ and $(x_j)_{j \in J}$, respectively. By the identity in \cite[(3.1)]{esary-1967-association},
  \begin{align}
    &\Cov(f((X_i)_{i \in I}), g((X_j)_{j \in J}))\\
    & \quad = \int_{-\infty}^{\infty}  \int_{-\infty}^{\infty} \Cov(1\{f((X_i)_{i \in I}) > s\}, 1\{g((X_j)_{j \in J}) > t\}) ds \medspace  dt. 
    \label{eq:cov-identity}
  \end{align}
  For each $s,t \in \R$, define
  \begin{equation}
    \gamma_s((x_i)_{i \in I}) \coloneqq 1\{f((x_i)_{i \in I}) > s\}, \quad \delta_t((x_j)_{j \in J}) \coloneqq 1\{g((x_j)_{j \in J}) > t\}.
  \end{equation}
  Both are binary-valued increasing functions. Therefore, the assumption~\eqref{eq:wPA-condition-assumption} gives
  \begin{equation}
    \Cov(\gamma_s((X_i)_{i \in I}), \delta_t((X_j)_{j \in J})) \geq 0.
  \end{equation}
  Hence, the right-hand side of \eqref{eq:cov-identity} is nonnegative, proving that $X$ is weakly positively associated.
\end{proof}

For $N = 3$, let $X$ be the binary random vector associated with $Y \sim p_{\alpha,\beta}$ given by \eqref{eq:curie-weiss}. By exchangeability, up to relabeling and interchanging $I$ and $J$, the only disjoint pairs $(I,J)$ that need to be considered in Lemma~\ref{lem:wPA-condition} are $(\{1\},\{2\})$ and $(\{1\},\{2,3\})$. Apart from the constant functions, the only binary-valued increasing function of $x_1$ is $x_1$ itself. Similarly, the nonconstant binary-valued increasing functions of $(x_2,x_3)$ are $x_2$, $x_3$, $x_2x_3$, and $x_2 \lor x_3$. Therefore, by exchangeability, $X$ is weakly positively associated if and only if all three quantities
\begin{equation}
  \Cov(X_1,X_2), \quad \Cov(X_1, X_2 X_3), \quad \Cov(X_1, X_2 \lor X_3)
\end{equation}
are nonnegative.
Set
\begin{equation}
  r \coloneqq e^\alpha, \quad s \coloneqq e^\beta, \quad Z \coloneqq e^{\psi(\alpha,\beta)} = 4 + 3r + s.
\end{equation}
A direct calculation then gives
\begin{align}
  Z^2 \Cov(X_1,X_2) &= 2s - r^2 - 1,\\
  Z^2 \Cov(X_1,X_2 X_3) &= 3s - 2r^2 - r,\\
  Z^2 \Cov(X_1,X_2 \lor X_3) &= s + r - 2.
\end{align}
Thus, weak positive association is equivalent to
\begin{equation}
  s \geq \max\left\{\frac{r^2 + 1}{2}, \frac{2r^2 + r}{3}, 2-r \right\}.
\end{equation}
The maximum on the right-hand side equals $2-r$ for $0<r\leq1$ and $(2r^2+r)/3$ for $r\geq1$, yielding the condition stated in Table~\ref{tb:cond-p-dependence}.

We next characterize the condition $p_{\alpha,\beta} \succism \nu_I$. Since $p_{\alpha,\beta}$ and $\nu_I$ have the same first-order marginals $\eta^*_{(1)}$, Remark~\ref{rem:ism-sm-equivalence} gives
\begin{equation}
  p_{\alpha,\beta} \succism \nu_I \iff p_{\alpha,\beta} \succsm \nu_I.
\end{equation}
Since both $p_{\alpha,\beta}$ and $\nu_I$ are exchangeable, \cite[Proposition~5(a) and Eq.~(15)]{meyer-2012-increasing} implies that
\begin{equation}
  p_{\alpha,\beta} \succsm \nu_I \iff \E_{p_{\alpha,\beta}}[\max\{|Y| - k, 0\}] \geq \E_{\nu_I}[\max\{|Y| - k, 0\}] \quad (k = 1,2).
  \label{eq:sm-equivalent-condition}
\end{equation}
We have
\begin{align}
  \max\{|Y| - 1, 0\} &= X_1X_2 + X_2X_3 + X_1X_3 - X_1X_2X_3, \\
  \max\{|Y| - 2, 0\} &= X_1X_2X_3,
\end{align}
and, for $i \neq j$,
\begin{equation}
  \E_{p_{\alpha,\beta}}[X_i X_j] = \frac{r + s}{Z}, \quad \E_{\nu_I}[X_i X_j] = \left(\frac{1 + 2r + s}{Z}\right)^2,
\end{equation}
as well as
\begin{equation}
  \E_{p_{\alpha,\beta}}[X_1X_2X_3] = \frac{s}{Z}, \quad \E_{\nu_I}[X_1X_2X_3] = \left(\frac{1 + 2r + s}{Z}\right)^3.
\end{equation}
Substituting these expressions into \eqref{eq:sm-equivalent-condition}, the condition for $k=1$ reduces to
\begin{equation}
  s \geq (r + 3)^{3/2} - 3r - 4,
\end{equation}
whereas the condition for $k=2$ reduces to
\begin{equation}
  s \geq \frac{3r^2 - 12r - 13 + \sqrt{(r + 3)^3(9r + 7)}}{10}.
\end{equation}
The first lower bound is larger for $0 < r \leq 1$, while the second is larger for $r \geq 1$, yielding the condition stated in Table~\ref{tb:cond-p-dependence}.

\subsection{Certified Bounds for the DPP Approximation Error}
\label{sec:numerical-bound-computation}

In general, computing $\KL(p^*\|\DPP)$ is a nonconvex optimization problem and is difficult to solve exactly. Here we use a simple branch-and-bound procedure to compute an approximate value together with lower and upper bounds. By Theorem~\ref{thm:DPP-error-dependence}, the diagonal entries of the marginal kernel $K$ may be fixed at
\begin{equation}
  \mu \coloneqq \eta_{\{1\}}(p_{\alpha,\beta}),
\end{equation}
so it remains to optimize over the three off-diagonal entries under the constraint $O \prec K \prec I$. We write
\begin{equation}
  K = \begin{pmatrix}
    \mu & a & b\\
    a & \mu & c\\
    b & c & \mu
  \end{pmatrix}.
\end{equation}
The corresponding DPP probability mass function $p_K$ is given, for example, by
\begin{align}
  p_K(\{1,2,3\}) &= \mu^3 - \mu(a^2 + b^2 + c^2) + 2abc, \label{eq:pK-123}\\
  p_K(\{1,2\}) &= \mu^2 - a^2 - p_K(\{1,2,3\}), \label{eq:pK-12}\\
  p_K(\{1\}) &= \mu - 2\mu^2 + a^2 + b^2 + p_K(\{1,2,3\}), \label{eq:pK-1}\\
  p_K(\emptyset) &= (1-\mu)^3 - (1-\mu)(a^2 + b^2 + c^2) - 2abc. \label{eq:pK-empty}
\end{align}
The remaining probabilities are obtained by permuting the indices.

The feasible region for $(a,b,c)$ is divided into rectangular boxes. On each box, we compute an upper bound $\overline{p}(A)$ for each probability $p_K(A)$. Since the probabilities in \eqref{eq:pK-123}--\eqref{eq:pK-empty} are polynomials in $(a,b,c)$, these bounds can be computed explicitly on each box. For any $K$ in the box, we then have
\begin{equation}
  \ell_\mathrm{pmf} \coloneqq \sum_{A \subset \{1,2,3\}} p^*(A) \log \frac{p^*(A)}{\overline{p}(A)} \leq \KL(p^*\|p_K).
  \label{eq:pmf-lower-bound}
\end{equation}

A complementary lower bound is obtained from the cardinality distribution. Let $\rho^*$ and $\rho_K$ denote the distributions of $|Y|$ under $p^*$ and $p_K$, respectively. By the data processing inequality applied to the map $A \mapsto |A|$,
\begin{equation}
    \KL(\rho^*\|\rho_K) \leq \KL(p^*\|p_K).
    \label{eq:card-lower-bound}
\end{equation}
Set $u \coloneqq a^2 + b^2 + c^2$ and $v \coloneqq abc$. Then
\begin{align}
  \rho_K (0) &= (1-\mu)^3 - (1-\mu)u - 2v, \\
  \rho_K (1) &= 3\mu(1-\mu)^2 + (2 - 3\mu)u + 6v, \\
  \rho_K (2) &= 3\mu^2(1-\mu) + (3\mu - 1)u - 6v, \\
  \rho_K (3) &= \mu^3 - \mu u + 2v.
\end{align}
Thus, each $\rho_K (m)$, $m = 0,1,2,3$, is an affine function of $(u,v)$. Together with the convexity of $-\log$, this implies that $\KL(\rho^*\|\rho_K)$ is jointly convex in $(u,v)$. Consequently, the supporting hyperplane of this convex function at $(u,v) = (0,0)$ yields a uniform lower bound, denoted by $\ell_\mathrm{card}$, on $\KL(\rho^*\|\rho_K)$ over the box, and hence also on $\KL(p^*\|p_K)$. For each box, we take $\ell \coloneqq \max\{\ell_\mathrm{pmf}, \ell_\mathrm{card}\}$ as the lower bound on $\KL(p^*\|p_K)$ over the box. Taking the minimum of $\ell$ over all boxes gives a certified lower bound on $\KL(p^*\|\DPP)$, and hence an upper bound on $\Delta$.

Conversely, for any feasible kernel $K$, the value $\KL(p^*\|p_K)$ directly gives an upper bound on $\KL(p^*\|\DPP)$:
\begin{equation}
  \KL(p^*\|\DPP) \leq \KL(p^*\|p_K).
\end{equation}
The smallest such value found during the branch-and-bound procedure therefore yields the corresponding lower bound on $\Delta$.

\end{appendices}


\bmhead{Acknowledgements}
Part of this work was supported by JSPS KAKENHI No.~JP26K02989, JP25H01494, and JP26K23861.

\bmhead{Data Availability}
No external datasets were used in this study.

\section*{Declarations}
\bmhead{Conflict of interest}
The authors declare no conflict of interest.


\begin{thebibliography}{39}
\ifx \bisbn   \undefined \def \bisbn  #1{ISBN #1}\fi
\ifx \binits  \undefined \def \binits#1{#1}\fi
\ifx \bauthor  \undefined \def \bauthor#1{#1}\fi
\ifx \batitle  \undefined \def \batitle#1{#1}\fi
\ifx \bjtitle  \undefined \def \bjtitle#1{#1}\fi
\ifx \bvolume  \undefined \def \bvolume#1{\textbf{#1}}\fi
\ifx \byear  \undefined \def \byear#1{#1}\fi
\ifx \bissue  \undefined \def \bissue#1{#1}\fi
\ifx \bfpage  \undefined \def \bfpage#1{#1}\fi
\ifx \blpage  \undefined \def \blpage #1{#1}\fi
\ifx \burl  \undefined \def \burl#1{\textsf{#1}}\fi
\ifx \doiurl  \undefined \def \doiurl#1{\url{https://doi.org/#1}}\fi
\ifx \betal  \undefined \def \betal{\textit{et al.}}\fi
\ifx \binstitute  \undefined \def \binstitute#1{#1}\fi
\ifx \binstitutionaled  \undefined \def \binstitutionaled#1{#1}\fi
\ifx \bctitle  \undefined \def \bctitle#1{#1}\fi
\ifx \beditor  \undefined \def \beditor#1{#1}\fi
\ifx \bpublisher  \undefined \def \bpublisher#1{#1}\fi
\ifx \bbtitle  \undefined \def \bbtitle#1{#1}\fi
\ifx \bedition  \undefined \def \bedition#1{#1}\fi
\ifx \bseriesno  \undefined \def \bseriesno#1{#1}\fi
\ifx \blocation  \undefined \def \blocation#1{#1}\fi
\ifx \bsertitle  \undefined \def \bsertitle#1{#1}\fi
\ifx \bsnm \undefined \def \bsnm#1{#1}\fi
\ifx \bsuffix \undefined \def \bsuffix#1{#1}\fi
\ifx \bparticle \undefined \def \bparticle#1{#1}\fi
\ifx \barticle \undefined \def \barticle#1{#1}\fi
\bibcommenthead
\ifx \bconfdate \undefined \def \bconfdate #1{#1}\fi
\ifx \botherref \undefined \def \botherref #1{#1}\fi
\ifx \url \undefined \def \url#1{\textsf{#1}}\fi
\ifx \bchapter \undefined \def \bchapter#1{#1}\fi
\ifx \bbook \undefined \def \bbook#1{#1}\fi
\ifx \bcomment \undefined \def \bcomment#1{#1}\fi
\ifx \oauthor \undefined \def \oauthor#1{#1}\fi
\ifx \citeauthoryear \undefined \def \citeauthoryear#1{#1}\fi
\ifx \endbibitem  \undefined \def \endbibitem {}\fi
\ifx \bconflocation  \undefined \def \bconflocation#1{#1}\fi
\ifx \arxivurl  \undefined \def \arxivurl#1{\textsf{#1}}\fi
\csname PreBibitemsHook\endcsname

\bibitem[\protect\citeauthoryear{Kulesza and Taskar}{2012}]{kulesza-2012-determinantal}
\begin{bbook}
\bauthor{\bsnm{Kulesza}, \binits{A.}},
\bauthor{\bsnm{Taskar}, \binits{B.}}:
\bbtitle{Determinantal Point Processes for Machine Learning}.
\bsertitle{Foundations and Trends in Machine Learning},
vol. \bseriesno{5}.
\bpublisher{Now Publishers Inc.},
\blocation{Hanover, MA, USA}
(\byear{2012}).
\doiurl{10.1561/2200000044}
\end{bbook}
\endbibitem

\bibitem[\protect\citeauthoryear{Macchi}{1975}]{macchi-1975-coincidence}
\begin{barticle}
\bauthor{\bsnm{Macchi}, \binits{O.}}:
\batitle{The coincidence approach to stochastic point processes}.
\bjtitle{Advances in Applied Probability}
\bvolume{7}(\bissue{1}),
\bfpage{83}--\blpage{122}
(\byear{1975})
\doiurl{10.2307/1425855}
\end{barticle}
\endbibitem

\bibitem[\protect\citeauthoryear{Kulesza and Taskar}{2011}]{kulesza-2011-learning}
\begin{bchapter}
\bauthor{\bsnm{Kulesza}, \binits{A.}},
\bauthor{\bsnm{Taskar}, \binits{B.}}:
\bctitle{Learning determinantal point processes}.
In: \beditor{\bsnm{Cozman}, \binits{F.G.}},
\beditor{\bsnm{Pfeffer}, \binits{A.}} (eds.)
\bbtitle{Proceedings of the Twenty-Seventh Conference on Uncertainty in Artificial Intelligence},
pp. \bfpage{419}--\blpage{427}.
\bpublisher{AUAI Press},
\blocation{Corvallis, OR, USA}
(\byear{2011})
\end{bchapter}
\endbibitem

\bibitem[\protect\citeauthoryear{Gartrell et~al.}{2016}]{gartrell-2016-bayesian}
\begin{bchapter}
\bauthor{\bsnm{Gartrell}, \binits{M.}},
\bauthor{\bsnm{Paquet}, \binits{U.}},
\bauthor{\bsnm{Koenigstein}, \binits{N.}}:
\bctitle{Bayesian low-rank determinantal point processes}.
In: \bbtitle{Proceedings of the 10th ACM Conference on Recommender Systems},
pp. \bfpage{349}--\blpage{356}.
\bpublisher{Association for Computing Machinery},
\blocation{New York, NY, USA}
(\byear{2016}).
\doiurl{10.1145/2959100.2959178}
\end{bchapter}
\endbibitem

\bibitem[\protect\citeauthoryear{Kojima and Komaki}{2016}]{kojima-2016-determinantal}
\begin{barticle}
\bauthor{\bsnm{Kojima}, \binits{M.}},
\bauthor{\bsnm{Komaki}, \binits{F.}}:
\batitle{Determinantal point process priors for {Bayesian} variable selection in linear regression}.
\bjtitle{Statistica Sinica}
\bvolume{26}(\bissue{1}),
\bfpage{97}--\blpage{117}
(\byear{2016})
\doiurl{10.5705/ss.202014.0161}
\end{barticle}
\endbibitem

\bibitem[\protect\citeauthoryear{Snoek et~al.}{2013}]{snoek-2013-determinantal}
\begin{bchapter}
\bauthor{\bsnm{Snoek}, \binits{J.}},
\bauthor{\bsnm{Zemel}, \binits{R.S.}},
\bauthor{\bsnm{Adams}, \binits{R.P.}}:
\bctitle{A determinantal point process latent variable model for inhibition in neural spiking data}.
In: \bbtitle{Advances in Neural Information Processing Systems},
vol. \bseriesno{26},
pp. \bfpage{1932}--\blpage{1940}.
\bpublisher{Curran Associates, Inc.},
\blocation{Red Hook, NY, USA}
(\byear{2013})
\end{bchapter}
\endbibitem

\bibitem[\protect\citeauthoryear{Batmanghelich et~al.}{2014}]{batmanghelich-2014-diversifying}
\begin{botherref}
\oauthor{\bsnm{Batmanghelich}, \binits{N.K.}},
\oauthor{\bsnm{Quon}, \binits{G.}},
\oauthor{\bsnm{Kulesza}, \binits{A.}},
\oauthor{\bsnm{Kellis}, \binits{M.}},
\oauthor{\bsnm{Golland}, \binits{P.}},
\oauthor{\bsnm{Bornn}, \binits{L.}}:
Diversifying Sparsity Using Variational Determinantal Point Processes
(2014).
\doiurl{10.48550/arXiv.1411.6307}
\end{botherref}
\endbibitem

\bibitem[\protect\citeauthoryear{Gillenwater et~al.}{2014}]{gillenwater-2014-expectation}
\begin{bchapter}
\bauthor{\bsnm{Gillenwater}, \binits{J.A.}},
\bauthor{\bsnm{Kulesza}, \binits{A.}},
\bauthor{\bsnm{Fox}, \binits{E.B.}},
\bauthor{\bsnm{Taskar}, \binits{B.}}:
\bctitle{Expectation-maximization for learning determinantal point processes}.
In: \bbtitle{Advances in Neural Information Processing Systems},
vol. \bseriesno{27},
pp. \bfpage{3149}--\blpage{3157}.
\bpublisher{Curran Associates, Inc.},
\blocation{Red Hook, NY, USA}
(\byear{2014})
\end{bchapter}
\endbibitem

\bibitem[\protect\citeauthoryear{Mariet and Sra}{2015}]{mariet-2015-fixed-point}
\begin{bchapter}
\bauthor{\bsnm{Mariet}, \binits{Z.}},
\bauthor{\bsnm{Sra}, \binits{S.}}:
\bctitle{Fixed-point algorithms for learning determinantal point processes}.
In: \beditor{\bsnm{Bach}, \binits{F.}},
\beditor{\bsnm{Blei}, \binits{D.}} (eds.)
\bbtitle{Proceedings of the 32nd International Conference on Machine Learning}.
\bsertitle{Proceedings of Machine Learning Research},
vol. \bseriesno{37},
pp. \bfpage{2389}--\blpage{2397}.
\bpublisher{PMLR},
\blocation{Lille, France}
(\byear{2015})
\end{bchapter}
\endbibitem

\bibitem[\protect\citeauthoryear{Kawashima and Hino}{2023}]{kawashima-2023-minorization}
\begin{botherref}
\oauthor{\bsnm{Kawashima}, \binits{T.}},
\oauthor{\bsnm{Hino}, \binits{H.}}:
Minorization-maximization for learning determinantal point processes.
Transactions on Machine Learning Research
(2023)
\end{botherref}
\endbibitem

\bibitem[\protect\citeauthoryear{Castella and Pesquet}{2026}]{castella-2026-kernel}
\begin{botherref}
\oauthor{\bsnm{Castella}, \binits{M.}},
\oauthor{\bsnm{Pesquet}, \binits{J.-C.}}:
Kernel matrix estimation of a determinantal point process from a finite set of samples: Properties and algorithms.
Transactions on Machine Learning Research
(2026)
\end{botherref}
\endbibitem

\bibitem[\protect\citeauthoryear{Affandi et~al.}{2014}]{affandi-2014-learning}
\begin{bchapter}
\bauthor{\bsnm{Affandi}, \binits{R.H.}},
\bauthor{\bsnm{Fox}, \binits{E.}},
\bauthor{\bsnm{Adams}, \binits{R.}},
\bauthor{\bsnm{Taskar}, \binits{B.}}:
\bctitle{Learning the parameters of determinantal point process kernels}.
In: \beditor{\bsnm{Xing}, \binits{E.P.}},
\beditor{\bsnm{Jebara}, \binits{T.}} (eds.)
\bbtitle{Proceedings of the 31st International Conference on Machine Learning}.
\bsertitle{Proceedings of Machine Learning Research},
vol. \bseriesno{32},
pp. \bfpage{1224}--\blpage{1232}.
\bpublisher{PMLR},
\blocation{Beijing, China}
(\byear{2014})
\end{bchapter}
\endbibitem

\bibitem[\protect\citeauthoryear{Urschel et~al.}{2017}]{urschel-2017-learning}
\begin{bchapter}
\bauthor{\bsnm{Urschel}, \binits{J.}},
\bauthor{\bsnm{Brunel}, \binits{V.-E.}},
\bauthor{\bsnm{Moitra}, \binits{A.}},
\bauthor{\bsnm{Rigollet}, \binits{P.}}:
\bctitle{Learning determinantal point processes with moments and cycles}.
In: \beditor{\bsnm{Precup}, \binits{D.}},
\beditor{\bsnm{Teh}, \binits{Y.W.}} (eds.)
\bbtitle{Proceedings of the 34th International Conference on Machine Learning}.
\bsertitle{Proceedings of Machine Learning Research},
vol. \bseriesno{70},
pp. \bfpage{3511}--\blpage{3520}.
\bpublisher{PMLR},
\blocation{Sydney, Australia}
(\byear{2017})
\end{bchapter}
\endbibitem

\bibitem[\protect\citeauthoryear{Gouri{\'e}roux and Lu}{2025}]{gourieroux-2025-simple}
\begin{botherref}
\oauthor{\bsnm{Gouri{\'e}roux}, \binits{C.}},
\oauthor{\bsnm{Lu}, \binits{Y.}}:
A Simple Estimator of the Correlation Kernel Matrix of a Determinantal Point Process
(2025).
\doiurl{10.48550/arXiv.2505.14529}
\end{botherref}
\endbibitem

\bibitem[\protect\citeauthoryear{Mariet and Sra}{2016}]{mariet-2016-kronecker}
\begin{bchapter}
\bauthor{\bsnm{Mariet}, \binits{Z.E.}},
\bauthor{\bsnm{Sra}, \binits{S.}}:
\bctitle{{Kronecker} determinantal point processes}.
In: \bbtitle{Advances in Neural Information Processing Systems},
vol. \bseriesno{29},
pp. \bfpage{2694}--\blpage{2702}.
\bpublisher{Curran Associates, Inc.},
\blocation{Red Hook, NY, USA}
(\byear{2016})
\end{bchapter}
\endbibitem

\bibitem[\protect\citeauthoryear{Gartrell et~al.}{2017}]{gartrell-2017-lowrank}
\begin{bchapter}
\bauthor{\bsnm{Gartrell}, \binits{M.}},
\bauthor{\bsnm{Paquet}, \binits{U.}},
\bauthor{\bsnm{Koenigstein}, \binits{N.}}:
\bctitle{Low-rank factorization of determinantal point processes}.
In: \bbtitle{Proceedings of the Thirty-First AAAI Conference on Artificial Intelligence},
vol. \bseriesno{31},
pp. \bfpage{1912}--\blpage{1918}
(\byear{2017}).
\doiurl{10.1609/aaai.v31i1.10869}
\end{bchapter}
\endbibitem

\bibitem[\protect\citeauthoryear{Hough et~al.}{2006}]{hough-2006-determinantal}
\begin{barticle}
\bauthor{\bsnm{Hough}, \binits{J.B.}},
\bauthor{\bsnm{Krishnapur}, \binits{M.}},
\bauthor{\bsnm{Peres}, \binits{Y.}},
\bauthor{\bsnm{Vir{\'a}g}, \binits{B.}}:
\batitle{Determinantal processes and independence}.
\bjtitle{Probability Surveys}
\bvolume{3},
\bfpage{206}--\blpage{229}
(\byear{2006})
\doiurl{10.1214/154957806000000078}
\end{barticle}
\endbibitem

\bibitem[\protect\citeauthoryear{Anari et~al.}{2016}]{anari-2016-mcmc}
\begin{bchapter}
\bauthor{\bsnm{Anari}, \binits{N.}},
\bauthor{\bsnm{Oveis~Gharan}, \binits{S.}},
\bauthor{\bsnm{Rezaei}, \binits{A.}}:
\bctitle{{Monte Carlo Markov Chain} algorithms for sampling {Strongly Rayleigh} distributions and determinantal point processes}.
In: \beditor{\bsnm{Feldman}, \binits{V.}},
\beditor{\bsnm{Rakhlin}, \binits{A.}},
\beditor{\bsnm{Shamir}, \binits{O.}} (eds.)
\bbtitle{29th Annual Conference on Learning Theory}.
\bsertitle{Proceedings of Machine Learning Research},
vol. \bseriesno{49},
pp. \bfpage{103}--\blpage{115}.
\bpublisher{PMLR},
\blocation{Columbia University, New York, New York, USA}
(\byear{2016})
\end{bchapter}
\endbibitem

\bibitem[\protect\citeauthoryear{Derezi{\'n}ski et~al.}{2019}]{derezinski-2019-exact}
\begin{bchapter}
\bauthor{\bsnm{Derezi{\'n}ski}, \binits{M.}},
\bauthor{\bsnm{Calandriello}, \binits{D.}},
\bauthor{\bsnm{Valko}, \binits{M.}}:
\bctitle{Exact sampling of determinantal point processes with sublinear time preprocessing}.
In: \bbtitle{Advances in Neural Information Processing Systems},
vol. \bseriesno{32}.
\bpublisher{Curran Associates, Inc.},
\blocation{Red Hook, NY, USA}
(\byear{2019})
\end{bchapter}
\endbibitem

\bibitem[\protect\citeauthoryear{Borcea et~al.}{2009}]{borcea-2009-negative}
\begin{barticle}
\bauthor{\bsnm{Borcea}, \binits{J.}},
\bauthor{\bsnm{Br{\"a}nd{\'e}n}, \binits{P.}},
\bauthor{\bsnm{Liggett}, \binits{T.M.}}:
\batitle{Negative dependence and the geometry of polynomials}.
\bjtitle{Journal of the American Mathematical Society}
\bvolume{22}(\bissue{2}),
\bfpage{521}--\blpage{567}
(\byear{2009})
\doiurl{10.1090/S0894-0347-08-00618-8}
\end{barticle}
\endbibitem

\bibitem[\protect\citeauthoryear{Brunel et~al.}{2017}]{brunel-2017-maximum}
\begin{botherref}
\oauthor{\bsnm{Brunel}, \binits{V.-E.}},
\oauthor{\bsnm{Moitra}, \binits{A.}},
\oauthor{\bsnm{Rigollet}, \binits{P.}},
\oauthor{\bsnm{Urschel}, \binits{J.}}:
Maximum Likelihood Estimation of Determinantal Point Processes
(2017).
\doiurl{10.48550/arXiv.1701.06501}
\end{botherref}
\endbibitem

\bibitem[\protect\citeauthoryear{Amari}{2001}]{amari-2001-information}
\begin{barticle}
\bauthor{\bsnm{Amari}, \binits{S.}}:
\batitle{Information geometry on hierarchy of probability distributions}.
\bjtitle{IEEE Transactions on Information Theory}
\bvolume{47}(\bissue{5}),
\bfpage{1701}--\blpage{1711}
(\byear{2001})
\doiurl{10.1109/18.930911}
\end{barticle}
\endbibitem

\bibitem[\protect\citeauthoryear{Mont{\'u}far and Rauh}{2014}]{montufar-2012-scaling}
\begin{barticle}
\bauthor{\bsnm{Mont{\'u}far}, \binits{G.F.}},
\bauthor{\bsnm{Rauh}, \binits{J.}}:
\batitle{Scaling of model approximation errors and expected entropy distances}.
\bjtitle{Kybernetika}
\bvolume{50}(\bissue{2}),
\bfpage{234}--\blpage{245}
(\byear{2014})
\doiurl{10.14736/kyb-2014-2-0234}
\end{barticle}
\endbibitem

\bibitem[\protect\citeauthoryear{Gatmiry et~al.}{2020}]{gatmiry-2020-testing}
\begin{bchapter}
\bauthor{\bsnm{Gatmiry}, \binits{K.}},
\bauthor{\bsnm{Aliakbarpour}, \binits{M.}},
\bauthor{\bsnm{Jegelka}, \binits{S.}}:
\bctitle{Testing determinantal point processes}.
In: \bbtitle{Advances in Neural Information Processing Systems},
vol. \bseriesno{33},
pp. \bfpage{12779}--\blpage{12791}.
\bpublisher{Curran Associates, Inc.},
\blocation{Red Hook, NY, USA}
(\byear{2020})
\end{bchapter}
\endbibitem

\bibitem[\protect\citeauthoryear{Amari and Nagaoka}{2000}]{amari-2000-methods}
\begin{bbook}
\bauthor{\bsnm{Amari}, \binits{S.}},
\bauthor{\bsnm{Nagaoka}, \binits{H.}}:
\bbtitle{Methods of Information Geometry}.
\bsertitle{Translations of Mathematical Monographs},
vol. \bseriesno{191}.
\bpublisher{American Mathematical Society and Oxford University Press},
\blocation{Providence, RI, USA and Oxford, UK}
(\byear{2000}).
\doiurl{10.1090/mmono/191}
\end{bbook}
\endbibitem

\bibitem[\protect\citeauthoryear{Hino and Yano}{2024}]{hino-2024-embedding}
\begin{barticle}
\bauthor{\bsnm{Hino}, \binits{H.}},
\bauthor{\bsnm{Yano}, \binits{K.}}:
\batitle{An embedding structure of determinantal point process}.
\bjtitle{Information Geometry}
\bvolume{7}(\bissue{2}),
\bfpage{523}--\blpage{542}
(\byear{2024})
\doiurl{10.1007/s41884-024-00156-x}
\end{barticle}
\endbibitem

\bibitem[\protect\citeauthoryear{Esary et~al.}{1967}]{esary-1967-association}
\begin{barticle}
\bauthor{\bsnm{Esary}, \binits{J.D.}},
\bauthor{\bsnm{Proschan}, \binits{F.}},
\bauthor{\bsnm{Walkup}, \binits{D.W.}}:
\batitle{Association of random variables, with applications}.
\bjtitle{The Annals of Mathematical Statistics}
\bvolume{38}(\bissue{5}),
\bfpage{1466}--\blpage{1474}
(\byear{1967})
\doiurl{10.1214/aoms/1177698701}
\end{barticle}
\endbibitem

\bibitem[\protect\citeauthoryear{Burton et~al.}{1986}]{burton-1986-invariance}
\begin{barticle}
\bauthor{\bsnm{Burton}, \binits{R.M.}},
\bauthor{\bsnm{Dabrowski}, \binits{A.R.}},
\bauthor{\bsnm{Dehling}, \binits{H.}}:
\batitle{An invariance principle for weakly associated random vectors}.
\bjtitle{Stochastic Processes and their Applications}
\bvolume{23}(\bissue{2}),
\bfpage{301}--\blpage{306}
(\byear{1986})
\doiurl{10.1016/0304-4149(86)90043-8}
\end{barticle}
\endbibitem

\bibitem[\protect\citeauthoryear{Fortuin et~al.}{1971}]{fortuin-1971-correlation}
\begin{barticle}
\bauthor{\bsnm{Fortuin}, \binits{C.M.}},
\bauthor{\bsnm{Kasteleyn}, \binits{P.W.}},
\bauthor{\bsnm{Ginibre}, \binits{J.}}:
\batitle{Correlation inequalities on some partially ordered sets}.
\bjtitle{Communications in Mathematical Physics}
\bvolume{22}(\bissue{2}),
\bfpage{89}--\blpage{103}
(\byear{1971})
\doiurl{10.1007/BF01651330}
\end{barticle}
\endbibitem

\bibitem[\protect\citeauthoryear{Joag-Dev and Proschan}{1983}]{joagdev-1983-negative}
\begin{barticle}
\bauthor{\bsnm{Joag-Dev}, \binits{K.}},
\bauthor{\bsnm{Proschan}, \binits{F.}}:
\batitle{Negative association of random variables with applications}.
\bjtitle{The Annals of Statistics}
\bvolume{11}(\bissue{1}),
\bfpage{286}--\blpage{295}
(\byear{1983})
\doiurl{10.1214/aos/1176346079}
\end{barticle}
\endbibitem

\bibitem[\protect\citeauthoryear{Pemantle}{2000}]{pemantle-2004-theory}
\begin{barticle}
\bauthor{\bsnm{Pemantle}, \binits{R.}}:
\batitle{Towards a theory of negative dependence}.
\bjtitle{Journal of Mathematical Physics}
\bvolume{41}(\bissue{3}),
\bfpage{1371}--\blpage{1390}
(\byear{2000})
\doiurl{10.1063/1.533200}
\end{barticle}
\endbibitem

\bibitem[\protect\citeauthoryear{Iyer and Bilmes}{2015}]{iyer-2015-submodular}
\begin{bchapter}
\bauthor{\bsnm{Iyer}, \binits{R.}},
\bauthor{\bsnm{Bilmes}, \binits{J.}}:
\bctitle{Submodular point processes with applications to machine learning}.
In: \beditor{\bsnm{Lebanon}, \binits{G.}},
\beditor{\bsnm{Vishwanathan}, \binits{S.V.N.}} (eds.)
\bbtitle{Proceedings of the Eighteenth International Conference on Artificial Intelligence and Statistics}.
\bsertitle{Proceedings of Machine Learning Research},
vol. \bseriesno{38},
pp. \bfpage{388}--\blpage{397}.
\bpublisher{PMLR},
\blocation{San Diego, CA, USA}
(\byear{2015})
\end{bchapter}
\endbibitem

\bibitem[\protect\citeauthoryear{Kawashima and Hino}{2025}]{kawashima-2025-family}
\begin{bchapter}
\bauthor{\bsnm{Kawashima}, \binits{T.}},
\bauthor{\bsnm{Hino}, \binits{H.}}:
\bctitle{A family of distributions of random subsets for controlling positive and negative dependence}.
In: \beditor{\bsnm{Li}, \binits{Y.}},
\beditor{\bsnm{Mandt}, \binits{S.}},
\beditor{\bsnm{Agrawal}, \binits{S.}},
\beditor{\bsnm{Khan}, \binits{E.}} (eds.)
\bbtitle{Proceedings of the 28th International Conference on Artificial Intelligence and Statistics}.
\bsertitle{Proceedings of Machine Learning Research},
vol. \bseriesno{258},
pp. \bfpage{64}--\blpage{72}.
\bpublisher{PMLR},
\blocation{Mai Khao, Thailand}
(\byear{2025})
\end{bchapter}
\endbibitem

\bibitem[\protect\citeauthoryear{Wainwright and Jordan}{2008}]{wainwright-2008-graphical}
\begin{barticle}
\bauthor{\bsnm{Wainwright}, \binits{M.J.}},
\bauthor{\bsnm{Jordan}, \binits{M.I.}}:
\batitle{Graphical models, exponential families, and variational inference}.
\bjtitle{Foundations and Trends in Machine Learning}
\bvolume{1}(\bissue{1--2}),
\bfpage{1}--\blpage{305}
(\byear{2008})
\doiurl{10.1561/2200000001}
\end{barticle}
\endbibitem

\bibitem[\protect\citeauthoryear{M{\"u}ller and Stoyan}{2002}]{muller-2002-comparison}
\begin{bbook}
\bauthor{\bsnm{M{\"u}ller}, \binits{A.}},
\bauthor{\bsnm{Stoyan}, \binits{D.}}:
\bbtitle{Comparison Methods for Stochastic Models and Risks}.
\bsertitle{Wiley Series in Probability and Statistics}.
\bpublisher{John Wiley \& Sons},
\blocation{Chichester, UK}
(\byear{2002})
\end{bbook}
\endbibitem

\bibitem[\protect\citeauthoryear{Horn and Johnson}{2012}]{horn-2012-matrix}
\begin{bbook}
\bauthor{\bsnm{Horn}, \binits{R.A.}},
\bauthor{\bsnm{Johnson}, \binits{C.R.}}:
\bbtitle{Matrix Analysis},
\bedition{2}nd edn.
\bpublisher{Cambridge University Press},
\blocation{Cambridge, UK}
(\byear{2012}).
\doiurl{10.1017/CBO9781139020411}
\end{bbook}
\endbibitem

\bibitem[\protect\citeauthoryear{Christofides and Vaggelatou}{2004}]{christofides-2004-connection}
\begin{barticle}
\bauthor{\bsnm{Christofides}, \binits{T.C.}},
\bauthor{\bsnm{Vaggelatou}, \binits{E.}}:
\batitle{A connection between supermodular ordering and positive/negative association}.
\bjtitle{Journal of Multivariate Analysis}
\bvolume{88}(\bissue{1}),
\bfpage{138}--\blpage{151}
(\byear{2004})
\doiurl{10.1016/S0047-259X(03)00064-2}
\end{barticle}
\endbibitem

\bibitem[\protect\citeauthoryear{Cover and Thomas}{2006}]{cover-1999-elements}
\begin{bbook}
\bauthor{\bsnm{Cover}, \binits{T.M.}},
\bauthor{\bsnm{Thomas}, \binits{J.A.}}:
\bbtitle{Elements of Information Theory},
\bedition{2}nd edn.
\bpublisher{John Wiley \& Sons},
\blocation{Hoboken, NJ, USA}
(\byear{2006}).
\doiurl{10.1002/047174882X}
\end{bbook}
\endbibitem

\bibitem[\protect\citeauthoryear{Meyer and Strulovici}{2012}]{meyer-2012-increasing}
\begin{barticle}
\bauthor{\bsnm{Meyer}, \binits{M.}},
\bauthor{\bsnm{Strulovici}, \binits{B.}}:
\batitle{Increasing interdependence of multivariate distributions}.
\bjtitle{Journal of Economic Theory}
\bvolume{147}(\bissue{4}),
\bfpage{1460}--\blpage{1489}
(\byear{2012})
\doiurl{10.1016/j.jet.2011.09.001}
\end{barticle}
\endbibitem

\end{thebibliography}

\end{document}